\documentclass[a4paper,11pt]{article}
\usepackage{amsmath}   % Advanced math typesetting
\usepackage{amsthm}    % Theorem formatting
\usepackage{amssymb}   % Extra math symbols
\usepackage{mathtools} % Mathematical tools to use with amsmath
\usepackage{mathrsfs}  % Support for the Raph Smith’s Formal Script font

\usepackage{csquotes} 

\usepackage{physics}   % Physics symbols and notation

\usepackage{enumerate} % Customizable enumerate environments

\usepackage{stmaryrd}  % Extra mathematical symbols

\usepackage{xspace}    

\usepackage{tikz}     

\usepackage[margin=1.0in]{geometry}

\usepackage[ruled,vlined]{algorithm2e}

\usepackage{cancel}   

\usepackage[backend=biber,doi=false,isbn=false,url=false,eprint=false]{biblatex}
\usepackage[hidelinks]{hyperref}  % Hyperlinks in the document

\newcommand\N{\mathbb N}

\newcommand\R{\mathbb R}
\newcommand\G{\mathscr G}

\newcommand\PP{\mathbb P}
\newcommand\E{\mathbb E}
\newcommand\T{\mathbb T}
\newcommand\F{\mathscr F}
\newcommand\X{\mathbf X}

\newcommand\Ind{\boldsymbol 1}

\DeclarePairedDelimiter{\ceil}{\lceil}{\rceil}
\DeclarePairedDelimiter{\floor}{\lfloor}{\rfloor}

\DeclarePairedDelimiterX{\infdivx}[2]{(}{)}{%
  #1\;\delimsize|\delimsize|\;#2%
}
\newcommand{\kld}[2]{\ensuremath{\mathbf{d}\infdivx{#1}{#2}}\xspace}

\newcommand{\PR}[1]{\PP\left(#1\right)}
\newcommand{\cPR}[2]{\PP\left(#1 \,\middle|\, #2\right)}

\newcommand{\EX}[1]{\E\left[#1\right]}
\newcommand{\cEX}[2]{\E\left[#1 \,\middle|\, #2 \right]}

\DeclareMathOperator*{\esssup}{ess\,sup}
\DeclareMathOperator*{\essinf}{ess\,inf}

\usepackage{xcolor}

\DeclarePairedDelimiterX\Ai[1]{\text{Ai}(}{)}{#1}

\theoremstyle{plain}
\newtheorem{theorem}{Theorem}[section]
\newtheorem*{theorem*}{Theorem}
\newtheorem{lemma}[theorem]{Lemma}
\newtheorem*{lemma*}{Lemma}
\newtheorem{corollary}[theorem]{Corollary}
\newtheorem*{corollary*}{Corollary}
\newtheorem{proposition}[theorem]{Proposition}
\newtheorem*{proposition*}{Proposition}
\newtheorem{fact}[theorem]{Fact}
\newtheorem*{fact*}{Fact}
\newtheorem{problem}[theorem]{Problem}
\newtheorem*{problem*}{Problem}

\theoremstyle{definition}
\newtheorem{definition}[theorem]{Definition}
\newtheorem{definition*}[theorem]{Definition}

\theoremstyle{remark}
\newtheorem{remark}[theorem]{Remark}
\newtheorem*{remark*}{Remark}

\numberwithin{equation}{section}

\allowdisplaybreaks

\begin{document}

\title{Sampling from the Gibbs measure of the continuous random energy model and the hardness threshold}
\author{
Fu-Hsuan~Ho\thanks{
  Institut de Math\'{e}matiques de Toulouse, CNRS UMR5219.
  \textit{Postal address:} Institut de Math\'{e}matiques de Toulouse,
  Universit\'{e} Toulouse 3 Paul Sabatier,
  118 Route de Narbonne, 31062 Toulouse Cedex 9, France.
  \textit{Email:} \texttt{fu-hsuan.ho AT math.univ-toulouse.fr}
  }
}
\date{\today}
\maketitle
\begin{abstract}
The continuous random energy model (CREM) is a toy model of disordered systems introduced by Bovier and Kurkova in 2004 based on previous work by Derrida and Spohn in the 80s. In a recent paper by Addario-Berry and Maillard, they raised the following question: what is the threshold $\beta_G$, at which sampling approximately the Gibbs measure at any inverse temperature $\beta>\beta_G$ becomes algorithmically hard? Here, sampling approximately means that the Kullback--Leibler divergence from the output law of the algorithm to the Gibbs measure is of order $o(N)$ with probability approaching $1$, as $N\rightarrow\infty$, and algorithmically hard means that the running time, the numbers of vertices queries by the algorithms, is beyond of polynomial order.

The present work shows that when the covariance function $A$ of the CREM is concave, for all $\beta>0$, a recursive sampling algorithm on a renormalized tree approximates the Gibbs measure with running time of order $O(N^{1+\varepsilon})$. For $A$ non-concave, the present work exhibits a threshold $\beta_G<\infty$ such that the following hardness transition occurs: a) For every $\beta\leq \beta_G$, the recursive sampling algorithm approximates the Gibbs measure with running time of order $O(N^{1+\varepsilon})$. b) For every $\beta>\beta_G$, a hardness result is established for a large class of algorithms. Namely, for any algorithm from this class that samples the Gibbs measure approximately, there exists $z>0$ such that the running time of this algorithm is at least $e^{zN}$ with probability approaching $1$. In other words, it is impossible to sample approximately in polynomial-time the Gibbs measure in this regime.
  
Additionally, we provide a lower bound of the free energy of the CREM that could hold its own value.

 
\paragraph{MSC2020 subject classifications:} 68Q17, 82D30, 60K35, 60J80. 
\end{abstract}

\section{Introduction} \label{sec:intro}

The continuous random energy model (CREM) is a toy model of a disordered system in statistical physics, i.e.~a model where the Hamiltonian -- the function that assigns energies to the states of the system -- is itself random. The CREM was introduced by Bovier and Kurkova~\cite{CREM04} based on previous work by Derrida and Spohn~\cite{DerridaSpohn}. 
%The precise definition of CREM is included in Section \ref{sec:def}.
Mathematically, the model is defined as follows. For a given integer $N\in\N$, the CREM is a centered Gaussian process $\X = (X_u)_{u\in\T_N}$ indexed by the binary tree $\T_N$ of depth $N$ with covariance function
\[\EX{X_vX_w} = N\cdot A\left(\frac{\abs{v\wedge w}}{N}\right), \quad \forall v,w\in \T_N.\]
Here, $\abs{v\wedge w}$ is the depth of the most recent common ancestor of $v$ and $w$, and the function $A$ is assumed to be a non-decreasing function defined on an interval $[0,1]$ such that $A(0)=0$ and $A(1)=1$. An essential quantity of this model is the Gibbs measure, which is a probability measure defined on the set of leaves $\partial\T_N$ where the weight of $v\in\partial\T_N$ is proportional to $e^{\beta X_v}$.

The present work consider the sampling problem of the Gibbs measure. We say that a (randomized) algorithm approximates the Gibbs measure if the Kullback--Leibler divergence from the output law of this algorithm to the Gibbs measure is of order $o(N)$ with probability approaching $1$.
The present work considers a recursive sampling algorithm that is similar to the one appearing in \cite{AB&M20} and \cite{KL_ejp22}. We shows that when the covariance function $A$ of the CREM is concave, for all $\beta>0$, the recursive sampling algorithm approximates the Gibbs measure with running time of order $O(N^{1+\varepsilon})$. Moreover, when $A$ is non-concave, we identify a threshold $\beta_G<\infty$ such that the following hardness transition occurs: a) For every $\beta\leq \beta_G$, the recursive sampling algorithm approximates the Gibbs measure with running time of order $O(N^{1+\varepsilon})$.
b) For every $\beta>\beta_G$, we prove a hardness result for a generic class of algorithms. Namely, there exists $\gamma>0$ such that for any algorithm in this class that approximates the Gibbs measure, the running time of this algorithm is at least $e^{\gamma N}$ with probability approaching $1$. 

\subsection{Definitions and notation} \label{sec:def}

Throughout this paper, we denote by $\mathbb{N}=\{1,2,\cdots\}$ the set of positive integer. For each pair of integers $n$ and $m$ such that $n\leq m$, we denote by $\llbracket n,m\rrbracket$ the set of integers between $n$ and $m$.

\paragraph{Binary tree.} 
Fixing $N\in\N$, we denote by $\T_N=\{\varnothing\}\cup\bigcup_{n=1}^N\{0,1\}^n$ the binary tree rooted at $\varnothing$. The depth of a vertex $v\in\T_N$ is denoted by $\abs{v}$. For any $v,w\in\T_N$, we write $v\leq w$ if $v$ is a prefix of $w$ and write $v<w$ if $v$ is a prefix of $w$ strictly shorter than $w$. In the following, for any $v\in\T_N$, we refer to any vertex $w$ with $w\leq v$ as an ancestor of $v$. For any $v\in\T_N$ and $n\in \llbracket 0,\abs{v}\rrbracket$, define $v[n]$ to be the ancestor of $v$ of depth $n$. 
%For any $v\in\T_N$ with prefix $w$, let $v\backslash w$ be the suffix of $v$ such that $v=w (v\backslash w)$. 
For all $v,w\in \T_N$, we denote by $v\wedge w$ the most recent common ancestor of $v$ and $w$.
We denote by $\partial\T_N$ the set of leaves of $\T_N$, and for any $v\in\T_N$, let $\T^v_n$ be the subtree of $\T_N$ rooted at $v$ with depth $n$.

\paragraph{Continuous random energy model.} Let $A$ be a non-decreasing function defined on an interval $[0,1]$ such that $A(0)=0$ and $A(1)=1$. For the sake of this paper, we assume that there exists a bounded Riemann integrable function $a$ such that $A$ for all $t\in [0,1]$,
\begin{align*}
A(t) = \int_0^t a(s) \dd{s}.
\end{align*}
We denote by $\hat{A}$ the concave hull of $A$ (see Figure~\ref{fig:CREM illustration}) and by $\hat{a}$ the right derivative of $\hat{A}$. Note that the $\hat{A}$ is also equals to the Riemann integral of $\hat{a}$, i.e., for all $t\in [0,1]$,
\begin{align*}
\hat{A}(t) = \int_0^t \hat{a}(s) \dd{s}.
\end{align*}

We now introduce the continuous random energy model (CREM). See Figure~\ref{fig:CREM illustration} for an illustration.
\begin{definition}
\label{def:CREM}
Given $N\in\N$, the continuous random energy model (CREM) is a centered Gaussian process $\X = (X_u)_{u\in\T_N}$ indexed by the binary tree $\T_N$ of depth $N$ with covariance function
\begin{align}
\EX{X_vX_w} = N\cdot A\left(\frac{\abs{v\wedge w}}{N}\right), \quad \forall v,w\in \T_N, \label{eq:cov of CREM}
\end{align}
where $\abs{v\wedge w}$ is the depth of the most recent common ancestor of $v$ and $w$.
\end{definition}

\begin{figure}[ht]
\centering
\begin{minipage}{0.6\textwidth}
\centering
\includegraphics[width=0.9\textwidth]{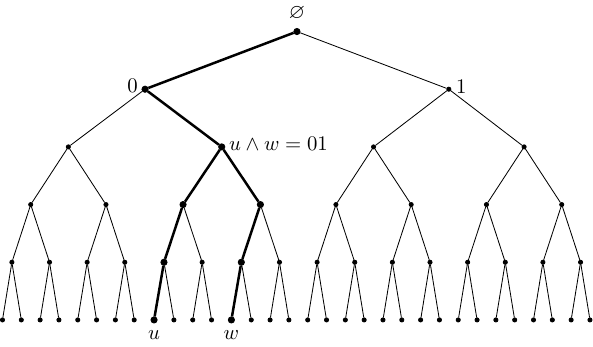}
\end{minipage}
\begin{minipage}{0.35\textwidth}
\centering
\includegraphics[width=0.85\textwidth]{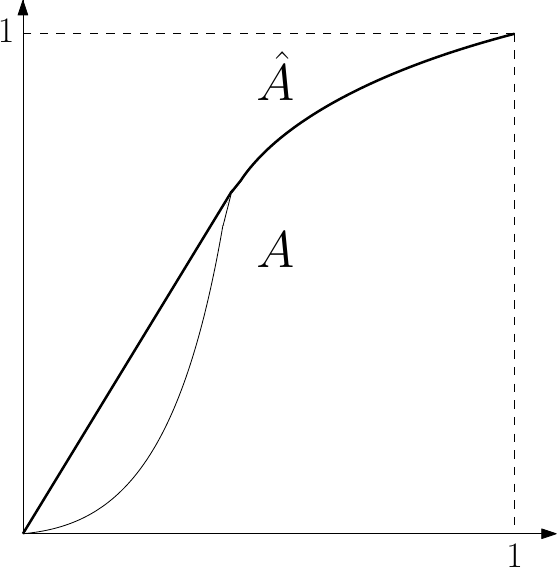}
\end{minipage}
%\caption{A binary tree.}
\caption{The covariance function of a CREM is determined by the underlying binary tree (left) and a function $A$ (right). The concave hull $\hat{A}$ of the function $A$ is also shown on the right, which determines the CREM free energy and the asymptotics of the maximum of the CREM.}
%\label{fig:binary_tree}
\label{fig:CREM illustration}
\end{figure}

Throughout this paper, we consider a sequence of CREM $(\mathbf{X}_N)_{N\in\N}$ defined on the same underlying probability space. For simplicity, we drop $N$ as long as it causes no ambiguity. 

\paragraph{Branching property.} The CREM can be viewed as an inhomogeneous binary branching random walk with Gaussian increments. In particular, it has the following branching property: let $(\F_k)_{k=0}^N$ be the natural filtration of the CREM. For any $u\in\T_N$ with $\abs{u}=n\in \llbracket 0,N\rrbracket$, call the process 
\begin{align*}
\mathbf{X}^u = (X^u_w)_{uw\in\T^u_{N-n}} 
\end{align*}
the CREM indexed by the subtree $\T^u_{N-n}$, where $X^u_w=X_{uw}-X_{u}$. For any $n\in\llbracket 0,N\rrbracket$, let $\mathbf{X}^{(n)}=(X^{(n)}_u)_{u\in\T_{N-n}}$ be a centered Gaussian process with covariance function 
\begin{align*}
\EX{X^{(n)}_{w_1}X^{(n)}_{w_2}} = N \cdot A\left(\frac{n+\abs{w_1\wedge w_2}}{N}\right), \quad \forall w_1,w_2\in\T_{N-n}.
\end{align*}
Then, the branching property states that collection of processes $\{\mathbf{X}^u:\abs{u}=n\}$ are independent and have the identical distribution of $\mathbf{X}^{(n)}$, and they are independent of $\F_n$. 

%However, unlike the homogeneous branching random walks, the law of each $\mathbf{X}^u$ depends now on the depth of $u$.

\paragraph{Partition function and Gibbs measure.} Given a subtree $\T_M^{v}$ rooted at $v$ and of depth $M\in \llbracket0, N-\abs{v}\rrbracket$, the \emph{Gibbs measure with inverse temperature $\beta>0$} is defined by 
\begin{align}
\mu_{\beta,M}^v(u) = \frac{1}{Z_{\beta,M}^v}e^{\beta X_u^v}, \quad \forall vu\in\T_M^v, \label{eq:Gibbs}
\end{align}
where 
\begin{align}
Z_{\beta,M}^v &= \sum_{vu\in\T_M^v} e^{\beta X_u^v}, \label{eq:par func}
\end{align}
is the partition function on the subtree $\T_M^v$. In particular, we adopt the conventions 
\begin{align*}
\mu_{\beta,M}=\mu_{\beta,M}^\varnothing \quad \text{and} \quad  Z_{\beta,M} = Z_{\beta,M}^\varnothing
\end{align*}
for any $M\in \llbracket 0,N\rrbracket$. For completeness, we also define $Z^{(n)}_{\beta,M} = \sum_{\abs{u}=M} e^{\beta X^{(n)}_u}$. 
\paragraph{Free energy and its lower bound.} For $v\in\T_N$, we refer to the logarithm of the partition function $\log Z_{\beta,M}^v$ as the free energy on the subtree $\T_M^v$. The free energy $F_\beta$ of the CREM is defined as follows, and $F_\beta$ admits an explicit expression.
\begin{align}
F_\beta \coloneqq \lim_{N\rightarrow\infty} \frac{1}{N}\EX{\log Z_{\beta,N}} = \int_0^1 f(\beta\sqrt{\hat{a}(s)})\dd{s}, \label{def:Fe}
\end{align}
where the function $f$ is defined as
\begin{align}
f(\beta)
=
\begin{cases}
\displaystyle
\log 2 + \frac{\beta^2}{2}, & \beta < \sqrt{2\log 2} \\
\\
\sqrt{2\log 2} \beta, & \beta \geq \sqrt{2\log 2}.
\end{cases}
\label{eq:BRW Fe}
\end{align}
For completeness, we include the proof of \eqref{def:Fe} in Fact~\ref{fact:Fe}.
When clear, we also simply refer to $F_\beta$ as the free energy. We introduce a related quantity $\tilde{F}_{\beta}$ defined as
\begin{align}
\tilde{F}_\beta \coloneqq \int_0^1 f(\beta\sqrt{a(s)})\dd{s}. \label{eq:tilde F_beta}
\end{align}
In Proposition~\ref{prop:FE comparison}, we show that $F_\beta\geq \tilde{F}_\beta$ and characterize the condition where the equality holds.

\paragraph{Algorithms.}
We follow the same definition of randomized algorithms as in \cite{AB&M20,KL_ejp22}, which also appeared in similar forms in \cite{PemantleSearch09}.
\begin{definition}[Algorithm]
\label{def:algor}
Let $N\in\N$. Let $\tilde{\F}_k$ be a filtration defined by 
\[
\tilde{\F}_k = \sigma\left(v(1),\ldots,v(k);\,X(v(1)),\ldots,X(v(k));\,U_1,\ldots,U_{k+1}\right)
\]
where $(U_k)_{k\geq 1}$ is a sequence of i.i.d. uniform random variables on $[0,1]$, independent of the continuous random energy model $\mathbf{X}$.
A random sequence $\mathrm{v}=(v(k))_{k\geq 0}$ taking values in $\mathbb{T}_N$ is called a \emph{(randomized) algorithm} if $v(0) = \varnothing$ and $v(k+1)$ is $\tilde{\F}_k$-measurable for every $k\geq 0$. We further suppose that there exists a stopping time $\tau$ with respect to the filtration $\tilde{\F}$ and such that $v(\tau) \in \partial \T_N$. We call $\tau$ the \emph{running time} and $v(\tau)$ the \emph{output} of the algorithm. The \emph{law of the output} is the (random) distribution of $v(\tau)$, conditioned on $\mathbf{X}$.
\end{definition}

\begin{remark}
Roughly speaking, the filtration $\tilde{\F} = (\tilde{\F}_k)_{k\geq 0}$ contains all the  information about everything the algorithm has queried so far, as well as the additional randomness needed to choose the next vertex.
\end{remark}

Throughout the paper, the notion of time complexity is given by the following definition.

\begin{definition}[Time complexity]
\label{def:time complexity}
Let $(\tau_N)$ be a sequence of running time corresponds to a sequence of algorithms indexed by $N$. Let $h:\mathbb{N}\rightarrow\mathbb{N}$ be a function. We say that the sequence of running times is of order $O(h(N))$ if almost surely, there exists $N_0\in\N$ such that $\tau_N \leq h(N)$. We say the running time is of polynomial order if there exists a polynomial $P(N)$ such that almost surely, there exists $N_0\in\N$.
\end{definition}

\begin{remark}
In the rest of the paper, when we say that the running time of an algorithm is of order $O(h(N))$, we implicitly assume that there is an underlying sequence of algorithms indexed by $N$, which we also refer to as an algorithm by abuse of notation. 
\end{remark}

\paragraph{Kullback--Leibler divergence.} Given two probability measures $P$ and $Q$ defined on a discrete space $\Omega$, the Kullback--Leibler divergence (also known as the relative entropy) from $Q$ to $P$ is defined by 
\begin{align} 
\label{def:KL}
\kld{P}{Q} &= \sum_{\omega\in \Omega} P(\omega)\cdot\log\left(\frac{P(\omega)}{Q(\omega)}\right).
\end{align}
From now on, we abbreviate the Kullback--Leibler divergence as the KL divergence.
\par
The notion of approximation in the present work is the following.
\begin{definition}
\label{def:approximation}
Let $(P_N)_{N\in\N}$ and $(Q_N)_{N\in\N}$ are two sequences of \emph{random} probability measures defined on a discrete space $\Omega$. We say that the sequence $(P_N)_{N\in\N}$ approximates the sequence $(Q_N)_{N\in\N}$ with probability approaching $1$ if
\begin{align*}
\lim_{N\rightarrow\infty}\PR{\frac{1}{N}\kld{P_N}{Q_N} < \varepsilon_N} = 0.
\end{align*}
\end{definition}

\begin{remark}
Note that Definition~\ref{def:approximation} is equivalent to saying that 
%there exists $\varepsilon_N\rightarrow 0$ such that
\begin{align*}
\frac{1}{N}\kld{P_N}{Q_N} \stackrel{\PP}{\rightarrow} 0, \quad \text{as $N\rightarrow\infty$.}
\end{align*}
\end{remark}

\subsection{Main results} \label{sec:main results}

Recall that $a$ is the derivative of $A$. By the Lebesgue criterion of Riemann integrability,  the function $a$ is continuous almost everywhere on $[0,1]$. If $A$ is non-concave, define the threshold 
\begin{align}
\beta_G = \frac{\sqrt{2\log 2}}{\esssup_{t\in \{A\neq\hat{A}\}} \sqrt{a(t)}}, \label{eq:threshold}
\end{align}
%where $t_G = \sup\{t\in [0,1]:\forall s\in [0,t],\, A(s)=\hat{A}(s)\}$. 
For completeness, we define $\beta_G=\infty$ when $A$ is concave.
We now state the main results. 

\subsubsection{Subcritical and critical regime \texorpdfstring{$\beta \leq \beta_G$}{beta small}: optimality of recursive sampling}

Fix $\beta>0$, $N\in\N$, and $M=M_N\in\llbracket 1,N\rrbracket$. Given a configuration of the continuous random energy model with depth $N$, consider the following algorithm:
\begin{center}
  \begin{algorithm}[H]
    \caption{Recursive sampling  on renormalized tree}
    \label{algor}
    %\SetAlgoLined
    set $v=\varnothing$\quad
    \While{$\abs{v}<N$} {
      sample $w$ with $\abs{w}=M\wedge (N-\abs{v})$ according to the Gibbs measure $\mu^v_{\beta,M\wedge (N-\abs{v})}$\quad
      replace $v$ with $vw$\quad
    }
    output $v$
  \end{algorithm}
\end{center}

\begin{remark}
\label{rem:algorithm}
This is the same algorithm as the one in \cite{KL_ejp22} except that now the law of the Gibbs measure $\mu^v_{\beta,M\wedge (N-\abs{v})}$ depends on the depth of $v$. Again, its running time is deterministic and bounded by $\lceil N/M \rceil 2^{M}$. The output law of Algorithm~\ref{algor} is a random probability measure $\mu_{\beta,M,N}$ on $\partial\T_N$ that is recursively defined as follows:
\begin{equation}
  \begin{split}
    \mu_{\beta,M,0}(\varnothing) &= 1 \\
    \mu_{\beta,M,N\wedge (k+1)M}(vw) &= \mu_{\beta,M,kM}(v)\cdot \mu^v_{\beta,M\wedge(N-kM)}(w) 
  \end{split}
  \label{eq:mu_decomposition}
\end{equation}
for all $\abs{v}=kM$, $\abs{w}=M\wedge(N-kM)$ and $ k\in \llbracket 0,\floor*{\frac{N}{M}}\rrbracket$.
It is not hard to see that
\begin{align}
\mu_{\beta,M,N}(u) 
= \frac{e^{\beta X_u}}{Z_{\beta,M,N}(u)}, \label{eq:algor dist}
\end{align}
where 
\begin{align*}
Z_{\beta,M,N}(u) = \prod_{k=1}^{\floor{N/M}} Z_{\beta,M\wedge (N-kM)}^{u[kM]}.
\end{align*}
\end{remark}

The first theorem states that the KL divergence from the output law of Algorithm~\ref{algor} to the Gibbs measure concentrates in the following sense.
\begin{theorem}[Concentration bounds] \label{main:concentration}
Let $\beta>0$, $N\in\N$ and $M\in\llbracket 1,N\rrbracket$. Then for all $p\geq 1$, there exists a constant $C_p>0$ depending only on $p$ such that
\begin{align*}
\frac{1}{N}\norm{\kld{\mu_{\beta,M,N}}{\mu_{\beta,N}}-\EX{\kld{\mu_{\beta,M,N}}{\mu_{\beta,N}}}}_p 
\leq \frac{\beta C_p}{\sqrt{M}}.
\end{align*}
\end{theorem}

Next, we show that with a suitable choice of $M_N$, the expectation of the KL divergence renormalized by $N$ converges to the difference between $F_\beta$ and $\tilde{F}_\beta$.

\begin{theorem}[Convergence of the KL divergence]
\label{main:convergence of KL}
Let $\beta>0$, $N\in\N$, and $M_N$ be a sequence such that $M_N\in\llbracket 1,N\rrbracket$ and $M_N\rightarrow\infty$. Let $\tilde{\mu}_{\beta,N} = \mu_{\beta,M_N,N}$ be the output law of Algorithm~\ref{algor}. Then,
\begin{align*}
\lim_{N\rightarrow\infty} \frac{1}{N} \EX{\kld{\tilde{\mu}_{\beta,N}}{\mu_{\beta,N}}} = F_\beta - \tilde{F}_\beta\geq 0,
\end{align*}
with equality holding if and only if $\beta\leq\beta_G$. 
\end{theorem}

As a corollary of Theorem~\ref{main:convergence of KL}, in the subcritical regime, with a good choice of $M_N$, the mean of the KL divergence divided by $N$ converges to $0$ when $N\rightarrow\infty$. Moreover, for $\varepsilon>0$, with a good choice of $M_N$, the running time is of $O(N^{1+\varepsilon})$.
\begin{corollary}[Efficient sampling]
  \label{main:eff sampling}
  Fix $\beta\in [0,\beta_G]$. Given $\varepsilon>0$, let $M_N=\floor{\varepsilon\log_2 N}\wedge N$ and
  $\tilde{\mu}_{\beta,N} = \mu_{\beta,M_N,N}$ be the output law of Algorithm~\ref{algor}. Then,  
  \begin{align}
  \label{eq:main.1}
  \lim_{N\rightarrow\infty} \frac{1}{N} \EX{\kld{\tilde{\mu}_{\beta,N}}{\mu_{\beta,N}}} = 0.
  \end{align}
Moreover, the running time is deterministic and of order $O(N^{1+\varepsilon})$.
\end{corollary}

\begin{remark}
Note that for $A$ concave, Corollary~\ref{main:eff sampling} yields that Algorithm~\ref{algor} approximates the Gibbs measure for all $\beta\in (0,\infty)$ as $\beta_G=\infty$ for $A$ concave.

Corollary~\ref{main:eff sampling} implies in particular that the algorithm approximates the Gibbs measure with probability approaching $1$. Indeed, if we choose, e.g., $\varepsilon_N = (\frac{1}{N} \EX{\kld{\tilde{\mu}_{\beta,N}}{\mu_{\beta,N}}})^{1/2}$, Corollary~\ref{main:eff sampling} and Markov's inequality then yield
\begin{align}
\PR{\frac{1}{N}\kld{\tilde{\mu}_{\beta,N}}{\mu_{\beta,N}} \leq \varepsilon_N} \geq 1 - \varepsilon_N^{1/2} \rightarrow 1, \quad \text{as} \quad N\rightarrow\infty. \label{eq:main.2}
\end{align}
%Corollary~\ref{main:eff sampling} also implies that with a good choice of $M_N$, the algorithm can be of ``almost-linear'' time in the sense that for all $h:\N\rightarrow \N$ such that $h_N = o(N)$ and $h_N\rightarrow\infty$, the running time is of $O(N h_N)$ by choosing $M_N=\floor{\log_2 h_N}$. 
%However, the algorithm cannot reach linear-time if we require \eqref{eq:main.1} to hold.
\end{remark}

We provide the proof of Corollary~\ref{main:eff sampling} below as it is short.

\begin{proof}[Proof of Corollary~\ref{main:eff sampling}]
Note that the choice of $M_N$ satisfies the assumption of Theorem~\ref{main:convergence of KL}, so the first statement follows directly from Theorem~\ref{main:convergence of KL}. Next, as mentioned in Remark~\ref{rem:algorithm}, the running time of Algorithm~\ref{algor} is deterministic and is bounded by $\ceil{N/M_N}2^{M_N}$. With our choice of $M_N$, we conclude that 
\begin{align*}
\ceil{N/M_N}2^{M_N}\leq N \cdot 2^{\varepsilon\log_2 N}\leq N^{1+\varepsilon}, 
\end{align*}
and the proof is completed.
\end{proof}

\subsubsection{Supercritical regime \texorpdfstring{$\beta > \beta_G$}{beta small}: hardness for generic algorithms}

Now we assume that $A$ is non-concave, so $\beta_G<\infty$. For $\beta > \beta_G$, we provide the following hardness result for the class of algorithms satisfying Definition~\ref{def:algor}.

\begin{theorem}[Hardness]
\label{main:hardness}
Suppose that $A$ is non-concave. Let $\beta > \beta_G$. For any algorithm satisfying Definition~\ref{def:algor} that approximates the Gibbs measure with probability approaching $1$, there exists $\gamma>0$ such that
\begin{align*}
\lim_{N\rightarrow\infty} \PR{\tau \geq e^{\gamma N}} = 1,
\end{align*}
where $\tau$ is the running time of the algorithm.
\end{theorem}

\subsection{Discussion and related work}

A natural way to sample from the Gibbs measure is via the Markov chain Monte Carlo (MCMC) method. In \cite{GREMMetropolis2020}, Nascimento and Fontes studied a Metropolis dynamics on the GREM, where the state space of this dynamics is the set of leaves. They showed that for all $\beta>0$, the spectral gap of the Metropolis dynamics decays exponentially to $0$ as $N\rightarrow\infty$ almost surely, which hinted that the MCMC method might not be the best way to approximate the Gibbs measure efficiently.

The current work is largely inspired by the previous work of Addario-Berry and Maillard \cite{AB&M20} on finding the near maximum (ground state) of the CREM. Bovier and Kurkova showed in \cite{CREM04} that the maximum of the CREM satisfies
\begin{align}
x_{GSE}\coloneqq\lim_{N\rightarrow\infty}\frac{1}{N}\EX{\max_{\abs{u}=N} X_u} = \sqrt{2\log 2}\int_0^1 \sqrt{\hat{a}(s)}\dd{s}. \label{eq:GSE}
\end{align}
With this result in mind, the problem that Addario-Berry and Maillard addressed can be phrased as the following optimization problem.
\begin{problem}
\label{prob:GSE}
For what kind of $A$ such that for all $\varepsilon>0$, there exists a polynomial-time algorithm that can find a vertex $\abs{u}=N$ such that $X_u\geq (x_{GSE}-\varepsilon)N$ with high probability?
\end{problem}
To respond to Problem~\ref{prob:GSE}, they showed the following phase transition: there exists a threshold 
\begin{align*}
x_*=\sqrt{2\log 2}\int_0^1 \sqrt{a(s)}\dd{s}
\end{align*}
such that for any $x<x_*$, there exists a linear time algorithm that finds $X_v\geq xN$ with high probability; for any $x>x_*$, there exists $z>0$ such that with high probability, it takes at least $e^{zN}$ queries to find $X_u\geq xN$. Since $x_{GSE}\geq x_*$ with equality holding if and only if $A$ is concave, the near maximum can be found if and only if $A$ is concave. Another remark is that their result correspond to the special case of our result where $\beta\rightarrow\infty$, and linear algorithm they proposed is similar to Algorithm~\ref{algor}.

%\paragraph{Optimization problem for mean-spin glasses.} 
Problem~\ref{prob:GSE} also appeared in the context of mean-field spin glass. It is known that a generalized Thouless--Anderson--Palmer approach proposed by Subag in  \cite{subagFreeEnergyLandscapes2018} gives a tree structure from the origin to the spin space when $\beta=\infty$. This picture allows Subag to show in \cite{Subag21} that for the full-RSB spherical spin glasses, a greedy type algorithm that exploits this tree structure gives an efficient way to find a near maximum of these models. On the other hand, it was conjectured by the physicists that when $\beta\rightarrow\infty$, the SK model also exhibits the full-RSB property (see \cite{mezardSpinGlassTheory1987}). By assuming this conjecture, Montanari solved Problem~\ref{prob:GSE} for the SK model in \cite{Montanari19} via the so-called approximate message passing (AMP) type algorithm, where one example of the AMP algorithms was Bolthausen's iteration scheme \cite{bolthausenIterativeConstructionSolutions2014} which solves the so-called TAP equation. Later, Alaoui, Montanari and Sellke \cite{alaouiOptimizationMeanfieldSpin2021}  extended Montanari's previous result to other mean-field spin glasses that do not exhibit the overlap gap property.

%This result was based on the previously known description established by Subag in \cite{subagFreeEnergyLandscapes2018}. In that paper, he generalized the Thouless--Anderson--Palmer (TAP) equation appeared in \cite{thoulessSolutionSolvableModel1977} 
%
%that the Gibbs measure of spherical spin glasses can be decomposed into a tree structure (ultrametric structure), which generalized the work of the physicists Thouless, Anderson and Palmer (TAP) in \cite{thoulessSolutionSolvableModel1977}.
%Subsequently, Subag showed in \cite{Subag21} that for a class of full-RSB spherical spin glasses, a greedy type algorithm that exploits this tree structure gives an efficient way to find a near  maximum of these models.

The problem of sampling from the Gibbs measure was also considered the context of mean-field spin glasses. This problem was usually attacked by introducing the Glauber dynamics, which also belongs to the MCMC method. For the Sherrington--Kirkpatrick model, physicists (see \cite{sompolinskyDynamicTheorySpinGlass1981,mezardSpinGlassTheory1987}) expected fast convergence to the Gibbs measure in the whole high temperature regime $\beta < 1$. Recently, it was shown by Bauerschmidt and Bodineau in \cite{bauerschmidtVerySimpleProof2019} and by Eldan, Koehler and Zeitouni in \cite{EKZ21} that fast mixing occurs when $\beta < 1/4$. Moreover, Eldan et al. showed in \cite{EKZ21} that the Gibbs measure satisfies a Poincar\'{e} inequality for the Dirichlet form of Glauber dynamics, so the Glauber dynamics mixes in $O(N^2)$ spin flips in total variation distance. Subsequently, this estimate was improved to $O(N\log N)$ by Anari et al. in \cite{anariEntropicIndependenceModified2021}. 

For spherical spin glasses, Gheissari and Jagannath in \cite{GJ19} that Langevin dynamics (continuum version of Glauber dynamics) have a polynomial spectral gap for $\beta$ small. On the other hand, Ben Arous and Jagannath proved in \cite{SpecGap18} that for $\beta$ sufficiently large, the mixing times of Glauber and Langevin dynamics are exponentially large in Ising and spherical spin glasses, respectively.

In \cite{EAMS22}, Alaoui, Montanari and Sellke proposed a non MCMC type algorithm based on the stochastic localization for the SK model. They showed that for $\beta < 1/2$, there exists an algorithm with complexity $O(N^2)$ with output law being close to the Gibbs measure in normalized Wasserstein distance. Moreover, for $\beta>1$, they established a hardness result for the stable algorithms, which means that the output law of these algorithms are stable under small random perturbation of the defining matrix of the SK model. The hardness result for $\beta>1$ was proven by utilizing the disorder chaos, which means for them that Wasserstein distance between the Gibbs measure and the perturbed Gibbs measure is bounded from below by a positive constant for arbitrary small random perturbation.

%\paragraph{MCMC on mean-field spin glasses} One natural strategy for sampling a complicated probability measure might be the Markov Chain Monte Carlo (MCMC) method. In the context of spin glasses, standard approaches for constructing such Markov chains include the Glauber dynamics and the Metropolis--Hastings dynamics, with the Gibbs measure serving as the stationary distribution of the Markov chain. Recent progress has been made in studying Glauber dynamics on mean-field spin glass models. For a detailed survey, see \cite{EAMS22}.

%\paragraph{MCMC on tree like models.} 

\paragraph{Overlap gap property and algorithmic hardness.} The overlap gap property, emerging from studying of mean-field spin glasses, seems to be an obstruction of many optimization algorithms for random structures. See Gamarnik \cite{gamarnikOverlapGapProperty2021} for a survey. In the context of the CREM, for a given $\beta>0$, the \emph{overlap distribution} is the limiting law (as $N\to\infty$) of the overlap $\frac{\abs{u\wedge w}}{N}$ of two vertices $u$ and $w$ sampled independently according to the Gibbs measure with inverse temperature $\beta$. The CDF of the limiting overlap distribution $\alpha_\beta:[0,1]\rightarrow [0,1]$ is defined as
\begin{align*}
\alpha_\beta(t) 
\coloneqq \lim_{N \rightarrow\infty} \EX{\sum_{\abs{u}=N}\sum_{\abs{w}=N} \mu_{\beta,N}(u)\mu_{\beta,N}(w)\Ind_{\abs{u\wedge w}/N\leq t}}, 
\end{align*}

The overlap gap property in the context of the CREM means that $\alpha_\beta(t)$ is equal to a constant strictly less than $1$ in an interval $[t_1,t_2]\subseteq [0,1]$. 
On the other hand, it is known in \cite{CREM04} that $\alpha_\beta(t)$ satisfies the following
\begin{align*}
\alpha_\beta(t) 
=
\begin{cases}
\displaystyle \frac{\sqrt{2\log 2}}{\beta\sqrt{\hat{a}(t)}}, & t \leq t_0(\beta), \\
\\
1, & t > t_0(\beta).
\end{cases} 
\end{align*}
When $A$ is concave, the CREM does not exhibits the overlap gap property for any $\beta>0$, which does not contradict the picture mentioned in the previous paragraph. On the other hand, when $A$ is non-concave, the CREM has the overlap gap property if and only if $\beta > \beta_G' = \sqrt{2\log 2}/\sqrt{\hat{a}(t_G)}$. Comparing with the hardness threshold $\beta_G$ defined in \eqref{eq:threshold}, we see $\beta_G < \beta_G'$, which means that some extra ingredients are needed to explain the algorithmic hardness we observe in the present work.

\paragraph{Further direction.} Corollary~\ref{main:eff sampling} implies that if $A$ is concave, for any $\beta\in [0,\infty)$, then the sequence of algorithm constructed from Algorithm~\ref{algor} can approximate the Gibbs measure in the sense of Definition~\ref{def:approximation}. One might ask whether a higher precision is achievable. Namely, let $\alpha\in [0,1)$. Given $\beta>0$, does there exist a sequence of algorithms with corresponding output laws $\tilde{\mu}_{\beta,N}$ such that
\begin{align*}
\lim_{N\rightarrow\infty} \frac{1}{N^{\alpha}}\kld{\tilde{\mu}_{\beta,N}}{\mu_{\beta,N}} = 0,
\end{align*}
with probability approaching $1$? Note that for a general class of branching random walks, with $\alpha=0$ and $\beta>\beta_c$, it is shown in \cite{KL_ejp22} that  with positive probability, this task has a running time of stretched exponential. The result in \cite{KL_ejp22} is derived from the fluctuation of the sampled path of supercritical Gibbs measure done by \cite{ChenMadauleMallein}.

\paragraph{Outline.} The paper is organized as follows. In Section~\ref{sec:decomposition of KL}, we prove that the KL divergence $\kld{\mu_{\beta,M,N}}{\mu_{\beta,N}}$ can be decomposed into weight sums of free energies on subtrees, and we also compute its expectation. This information readily allows us to prove Theorem~\ref{main:concentration} which is provided in Section~\ref{sec:proof of concentration}. Building on the decomposition of the KL divergence provided in Section~\ref{sec:proof of concentration}, we study in Section~\ref{sec:asymptotics of KL} the renormalized limit of $\EX{\kld{\mu_{\beta,M,N}}{\mu_{\beta,N}}}$. This leads to the proof of Theorem~\ref{main:convergence of KL} which is provided at the end of the introduction of Section~\ref{sec:asymptotics of KL}. In Section~\ref{sec:supercritical Gibbs}, we show that for $A$ non-concave and $\beta>\beta_G$, the Gibbs measure tends to sample a rare event. Based on this observation, Theorem~\ref{main:hardness} is proven in Section~\ref{sec:hardness}, where the details are provided at the end of the introduction of Section~\ref{sec:hardness}. In Appendix~\ref{ann:lowbnd of Fe}, we provide a lower bound of the free energy $F_\beta$ that may be of independent interest. Finally, in Appendix~\ref{ann:proof of lem:sandwich}, we provide the details of the proof of Lemma~\ref{lem:sandwich}.

\section{Decomposition of the KL divergence \texorpdfstring{$\kld{\mu_{\beta,M,N}}{\mu_{\beta,N}}$}{dKL}} \label{sec:decomposition of KL}

In this section, we provide in the following proposition a simple decomposition of the KL divergence $\kld{\mu_{\beta,M,N}}{\mu_{\beta,N}}$ in terms of a sum of free energies on subtrees.

\begin{proposition}\label{prop:KL decomp}
For all $\beta>0$ and for any two integers $M,N\in\N$ such that $M\leq N$, we have
\begin{align*}
\kld{\mu_{\beta,M,N}}{\mu_{\beta,N}}
= \log Z_{\beta,N} - \sum_{k=0}^{\floor{N/M}}\sum_{\abs{u}=kM} \mu_{\beta,M,kM}(u)\cdot \log Z^{u}_{\beta,M\wedge (N-kM)}.
\end{align*}
\end{proposition}
\begin{proof}
By \eqref{def:KL} the definition of the KL divergence,
\begin{align}
&\kld{\mu_{\beta,M,N}}{\mu_{\beta,N}} \nonumber \\
&= \sum_{\abs{u}=N} \mu_{\beta,M,N}(u)\cdot (\log Z_{\beta,N} - \log Z_{\beta,M,N}(u)) &&  \text{(By \eqref{eq:Gibbs} and \eqref{eq:algor dist})} \nonumber \\
&= \sum_{\abs{u}=N} \mu_{\beta,M,N}(u)\cdot \Big(\log Z_{\beta,N} - \sum_{k=0}^{\floor{N/M}} \log Z^{u[kM]}_{\beta,M\wedge (N-kM)}\Big) \nonumber \\
&= \log Z_{\beta,N} - \sum_{\abs{u}=N}\sum_{k=0}^{\floor{N/M}} \mu_{\beta,M,N}(u)\cdot \log Z^{u[kM]}_{\beta,M\wedge (N-kM)} \nonumber \\
&= \log Z_{\beta,N} - \sum_{k=0}^{\floor{N/M}}\sum_{\abs{u}=N} \mu_{\beta,M,N}(u)\cdot \log Z^{u[kM]}_{\beta,M\wedge (N-kM)} \nonumber \\
&= \log Z_{\beta,N} - \sum_{k=0}^{\floor{N/M}} \sum_{\abs{w}=kM}\sum_{\abs{w'}=N-kM} \mu_{\beta,M,N}(ww')\cdot \log Z^{w}_{\beta,M\wedge (N-kM)} \nonumber \\
&= \log Z_{\beta,N} - \sum_{k=0}^{\floor{N/M}}\sum_{\abs{w}=kM} \mu_{\beta,M,kM}(w)\cdot \log Z^{w}_{\beta,M\wedge (N-kM)}, && \text{(By \eqref{eq:mu_decomposition})} \nonumber
\end{align}
the proof is completed.
\end{proof}

The next proposition asserts that the expectation of the KL divergence $\kld{\mu_{\beta,M,N}}{\mu_{\beta,N}}$ can be written as the difference between the free energy of the CREM and the sum of free energies on the subtrees.

\begin{proposition} 
\label{prop:KL expectation}
For all $\beta>0$ and for any two integers $M,N\in\N$ such that $M\leq N$, the expectation of the KL divergence from $\mu_{\beta,M,N}$ to $\mu_{\beta,N}$ admits the following decomposition.
\begin{align*}
\EX{\kld{\mu_{\beta,M,N}}{\mu_{\beta,N}}}
= \EX{\log Z_{\beta,N}} 
- \sum_{k=1}^{\floor{N/M}} \EX{\log Z^{(kM)}_{\beta,M\wedge (N-kM)}}.
\end{align*}
\end{proposition}

\begin{proof}
Combining Proposition~\ref{prop:KL decomp}, the branching property and the law of iterated expectation, we have
\begin{align}
&\EX{\kld{\mu_{\beta,M,N}}{\mu_{\beta,N}}} \nonumber \\
&= \EX{\log Z_{\beta,N}} - \EX{\sum_{k=0}^{\floor{N/M}}\sum_{\abs{u}=kM} \mu_{\beta,M,kM}(u)\cdot \log Z^{u}_{\beta,M\wedge (N-kM)}} \nonumber \\
&= \EX{\log Z_{\beta,N}} - \EX{\sum_{k=0}^{\floor{N/M}}\sum_{\abs{u}=kM} \mu_{\beta,M,kM}(u)\cdot \cEX{\log Z^{u}_{\beta,M\wedge (N-kM)}}{\F_{kM}}} \nonumber \\
&= \EX{\log Z_{\beta,N}} - \sum_{k=0}^{\floor{N/M}}\EX{\sum_{\abs{u}=kM} \mu_{\beta,M,kM}(u)} \EX{\log Z^{(kM)}_{\beta,M\wedge (N-kM)}} \nonumber \\
&= \EX{\log Z_{\beta,N}} - \sum_{k=0}^{\floor{N/M}}\EX{\log Z^{(kM)}_{\beta,M\wedge (N-kM)}}, \label{eq:mu is prob}
\end{align}
where \eqref{eq:mu is prob} follows from the fact that $\mu_{\beta,M,kM}$ is a probability measure.
\end{proof}

\subsection{Proof of Theorem~\ref{main:concentration}} \label{sec:proof of concentration}
%\TODO{In fact, $L^1$ convergence is sufficient for our purpose!}
The following argument is similar to the proof of (1.9) in \cite{KL_ejp22}, where the difference is that we use the concentration inequalities of free energies to control certain terms. 

Let $p\geq 1$. By Proposition~\ref{prop:KL decomp} and Minkowski's inequality, we have
\begin{align}
&\frac{1}{N}\norm{\kld{\mu_{\beta,M,N}}{\mu_{\beta,N}}-\EX{\kld{\mu_{\beta,M,N}}{\mu_{\beta,N}}}}_p \nonumber \\
&\leq \frac{1}{N}\sum_{k=0}^{\floor{N/M}} \norm{\sum_{\abs{u}=kM}\mu_{\beta,M,kM}(u) \cdot \left(\log Z^u_{\beta,M\wedge (N-kM)}-\EX{\log Z^u_{\beta,M\wedge (N-kM)}} \right)}_p. \label{eq:Minkowski uppbnd}
%&\leq \frac{1}{N}\EX{\sum_{k=0}^{\floor{N/M}}\sum_{\abs{u}=kM}\abs{\mu_{\beta,M,kM}(u)}\cdot \abs{\log Z^u_{\beta,M\wedge (N-kM)}-\EX{\log Z^u_{\beta,M\wedge (N-kM)}}}} \\
%&= \frac{1}{N}\EX{\sum_{k=0}^{\floor{N/M}}\sum_{\abs{u}=kM}\mu_{\beta,M,kM}(u)\cdot \abs{\log Z^u_{\beta,M\wedge (N-kM)}-\EX{\log Z^u_{\beta,M\wedge (N-kM)}}}} \\
%&= \frac{1}{N}\sum_{k=0}^{\floor{N/M}}\norm{\log Z^{(kM)}_{\beta,M\wedge (N-kM)}-\EX{\log Z^{(kM)}_{\beta,M\wedge (N-kM)}}}_p.
\end{align}
Applying Jensen's inequality to $\mu_{\beta,M,kM}$ and the fact that $x\mapsto\abs{x}^p$ is convex for all $p\geq 1$, we obtain
\begin{align}
&\EX{\abs{\sum_{\abs{u}=kM}\mu_{\beta,M,kM}(u) \cdot \left(\log Z^u_{\beta,M\wedge (N-kM)}-\EX{\log Z^u_{\beta,M\wedge (N-kM)}} \right)}^p} \nonumber \\
&\leq \EX{\sum_{\abs{u}=kM}\mu_{\beta,M,kM}(u) \cdot \abs{\log Z^u_{\beta,M\wedge (N-kM)}-\EX{\log Z^u_{\beta,M\wedge (N-kM)}} }^p}. \label{eq:weight sum uppbnd}
\end{align}
Then by the law of iterated expectation and the branching property, the expectation \eqref{eq:weight sum uppbnd} above equals 
\begin{align}
&\EX{\sum_{\abs{u}=kM}\mu_{\beta,M,kM}(u) \cdot \cEX{\abs{\log Z^u_{\beta,M\wedge (N-kM)}-\EX{\log Z^u_{\beta,M\wedge (N-kM)}} }^p}{\F_{kM}}} \nonumber \\
&= \EX{\abs{\log Z^{(kM)}_{\beta,M\wedge (N-kM)}-\EX{\log Z^{(kM)}_{\beta,M\wedge (N-kM)}} }^p}. \label{eq:Lp norm}
\end{align}
Now, the concentration inequality of free energies (see, Theorem 1.2 in \cite{SK_Panchenko}) implies that for all $p\geq 1$,
\begin{align}
&\EX{\abs{\log Z^{(kM)}_{\beta,M\wedge (N-kM)}-\EX{\log Z^{(kM)}_{\beta,M\wedge (N-kM)}}}^p} \nonumber \\
&\leq \int_0^\infty 2 \exp(-\frac{x^{2/p}}{4\beta^2N\Big(A\left(\frac{kM+M\wedge (N-kM)}{N}\right)-A\left(\frac{kM}{N}\right)\Big)}) \dd{x} \nonumber \\
&= \beta^p N^{p/2}\left(A\left(\frac{kM+M\wedge (N-kM)}{N}\right)-A\left(\frac{kM}{N}\right)\right)^{p/2} \cdot \underbrace{2^{p/2+1} p \int_0^\infty \exp(-\frac{y^2}{2}) y^{p-1}\dd{y}}_{\eqqcolon C_1(p)} \nonumber \\
&= \beta^p N^{p/2}\left(A\left(\frac{kM+M\wedge (N-kM)}{N}\right)-A\left(\frac{kM}{N}\right)\right)^{p/2} \cdot C_1(p). \label{eq:log Z concentration}
\end{align}
Combining \eqref{eq:Minkowski uppbnd}, \eqref{eq:Lp norm} and \eqref{eq:log Z concentration} and letting $C_2(p)=C_1(p)^{1/p}$, we derive that
\begin{align}
&\frac{1}{N}\norm{\kld{\mu_{\beta,M,N}}{\mu_{\beta,N}}-\EX{\kld{\mu_{\beta,M,N}}{\mu_{\beta,N}}}}_p \nonumber \\
&\leq \frac{\beta C_2(p)}{\sqrt{N}} \sum_{k=0}^{\floor{N/M}} \sqrt{\Big(A\left(\frac{kM+M\wedge (N-kM)}{N}\right)-A\left(\frac{kM}{N}\right)\Big)} \nonumber \\
&\leq \frac{\beta C_2(p)}{\sqrt{N}}\sqrt{\underbrace{\sum_{k=0}^{\floor{N/M}} \left(A\left(\frac{kM+M\wedge (N-kM)}{N}\right)-A\left(\frac{kM}{N}\right)\right)}_{=1}} \sqrt{\sum_{k=0}^{\floor{N/M}} 1}  \label{eq:C-S uppbnd} \\
&=\frac{\beta C_2(p)}{\sqrt{N}} \sqrt{\ceil*{\frac{N}{M}}} \nonumber \\
&\leq \frac{\beta C_2(p)\sqrt{2}}{\sqrt{M}}, \label{eq:sqrt 2/M} 
\end{align}
where \eqref{eq:C-S uppbnd} is derived from the Cauchy--Schwarz inequality, and \eqref{eq:sqrt 2/M} follows from bounding $\sqrt{\floor{N/M}/N}$ by $\sqrt{2/M}$. By choosing $C_p = C_2(p)\sqrt{2}$, the proof of Theorem~\ref{main:concentration} is completed.

\section{Asymptotics of the KL divergence \texorpdfstring{$\kld{\mu_{\beta,M,N}}{\mu_{\beta,N}}$}{}} \label{sec:asymptotics of KL}

The goal of this section is to prove Theorem~\ref{main:convergence of KL}. In view of \eqref{def:Fe} and Proposition~\ref{prop:KL expectation}, it remains to show the following proposition.
\begin{proposition}
\label{prop:renorm Fe}
Let $M_N$ be a sequence such that $M_N \in \llbracket 1,N\rrbracket$ and $M_N\rightarrow\infty$. Then,
\begin{align*}
\lim_{N\rightarrow\infty}\frac{1}{N}\sum_{k=1}^{\floor{N/M_N}} \EX{\log Z^{(kM_N)}_{\beta,M_N\wedge (N-kM_N)}} = \tilde{F}_\beta.
\end{align*}
\end{proposition}

The proof of Proposition~\ref{prop:renorm Fe} is postponed to Section~\ref{sec:proof of prop:renorm Fe}. 
Conditioned on Proposition~\ref{prop:renorm Fe}, we are now ready to prove Theorem~\ref{main:convergence of KL}.

\begin{proof}[Proof of Proposition~\ref{main:convergence of KL}]
Fix $\beta < \beta_G$. Let $M_N$ be a sequence such that $M_N\in\llbracket 1,N\rrbracket$, and $M_N\rightarrow\infty$. By Fact~\ref{fact:Fe} and Proposition~\ref{prop:renorm Fe}, 
\begin{align}
\lim_{N\rightarrow\infty} \frac{1}{N}\EX{\kld{\mu_{\beta,M_N,N}}{\mu_{\beta,N}}}
&= \lim_{N\rightarrow\infty} \frac{1}{N}\EX{\log Z_{\beta,N}} 
- \lim_{N\rightarrow\infty} \frac{1}{N}\sum_{k=1}^{\floor{N/M_N}} \EX{\log Z^{(kM_N)}_{\beta,M_N\wedge (N-kM_N)}}
 \nonumber \\
&= F_\beta - \tilde{F}_\beta. \nonumber
\end{align}
Now, by Proposition~\ref{prop:FE comparison}, $F_\beta - \tilde{F}_\beta = 0$ for all $\beta\leq \beta_G$ and $F_\beta - \tilde{F}_\beta >0 $ for all $\beta> \beta_G$. This completes the proof.
\end{proof}

\subsection{Proof of Proposition~\ref{prop:renorm Fe}} \label{sec:proof of prop:renorm Fe}

The proof of Proposition~\ref{prop:renorm Fe} is based on the following lemma.

\begin{lemma}
\label{lem:sandwich}
Let $M_N$ be a sequence such that $M_N\in \llbracket 1,N\rrbracket$ and $M_N\rightarrow\infty$.
For all $k\in\llbracket 0,\floor{N/M_N}\rrbracket$, define
\begin{align}
a_k^- \coloneqq \essinf_{t\in [\frac{kM_N}{N},\frac{(k+1)M_N}{N}]} a(t)
\quad \text{and}
\quad a_k^+ \coloneqq \esssup_{t\in [\frac{kM_N}{N},\frac{(k+1)M_N}{N}]} a(t).
\end{align} 
Then for all $\varepsilon>0$, there exists $N_0\in\N$ such that for all $N\geq N_0$, 
\begin{align*}
f(\beta\sqrt{a}^-_k) - \varepsilon\beta\sqrt{a}^-_k
\leq \frac{1}{M_N}\EX{\log Z^{(kM)}_{\beta,M}}
\leq 
f(\beta\sqrt{a}^+_k) + \varepsilon\beta\sqrt{a}^+_k,
\end{align*}
for all $k\in\llbracket 0,\floor{N/M_N}\rrbracket$. 
\end{lemma}
The proof of Lemma~\ref{lem:sandwich} is based on comparing the free energy of the CREM with the free energy of the so-called branching random walk, which is a CREM with $A$ equal to the identity function. While the proof of Lemma~\ref{lem:sandwich} is rather standard, the proof requires some standard properties of the free energy of the branching random walk, so we postpone the proof of Lemma~\ref{lem:sandwich} to Appendix~\ref{ann:proof of lem:sandwich}.

We now proceed to the proof of Proposition~\ref{prop:renorm Fe}.

\begin{proof}[Proof of Proposition~\ref{prop:renorm Fe}]
Fix $\varepsilon>0$. Fix $M_N$ being a sequence such that $M_N\in \llbracket 1,N\rrbracket$ and $M_N\rightarrow\infty$. We denote $K_N=\floor{N/M_N}$ for simplicity. First of all, note that
\begin{align}
\frac{1}{N}\sum_{k=1}^{K_N} \EX{\log Z_{\beta,M_N}^{(kM_N)}}
= \frac{1}{N}\sum_{k=1}^{K_N-1} \EX{\log Z_{\beta,M_N}^{(kM_N)}} + \frac{1}{N}\EX{\log Z_{\beta,N-K_N M_N}^{(K_N M_N)}}. \label{eq:tilde Fe sum}
\end{align}
We claim that the second term of \eqref{eq:tilde Fe sum} converges to $0$. For any $\abs{u}=N-\floor{N/M_N}M_N$,
\begin{align*}
\EX{\log Z_{\beta,N-\floor{N/M_N}M_N}^{(\floor{N/M_N}M_N)}} \geq \EX{\beta X_u^{(\floor{N/M_N}M_N)}} = 0.
\end{align*}
Now, we turn to the upper bound. By Jensen's inequality,
\begin{align}
\frac{1}{N} \EX{\log Z_{\beta,N-\floor{N/M_N}M_N}^{\floor{N/M_N}M_N}}
&\leq \frac{1}{N}\log \EX{Z_{\beta,N-\floor{N/M_N}M_N}^{(\floor{N/M_N}M_N)}} \nonumber \\
&= \log 2 \left(1-\floor*{\frac{N}{M_N}}\frac{M_N}{N}\right) 
+ \left(A(1) - A\left(\floor*{\frac{N}{M_N}}\frac{M_N}{N}\right)\right)
\rightarrow 0, \nonumber
\end{align}
as $N\rightarrow\infty$, which proves that the second term of \eqref{eq:tilde Fe sum} converges to $0$.

It remains to show that the first term of \eqref{eq:tilde Fe sum} converges to $\tilde{F}_\beta$. By Lemma~\ref{lem:sandwich},
\begin{align}
\frac{M_N}{N}\sum_{k=1}^{K_N-1} f(\beta\sqrt{a}^-_k) - \varepsilon\beta\sqrt{a}^-_k
\leq \frac{1}{N} \sum_{k=1}^{K_N-1} \EX{\log Z^{(kM)}_{\beta,M}}
\leq 
\frac{M_N}{N}\sum_{k=1}^{K_N-1} f(\beta\sqrt{a}^+_k) + \varepsilon\beta\sqrt{a}^+_k.
\label{eq:second sandwich}
\end{align}
Because the function $a(\cdot)$ is Riemann integrable and $f(\beta \sqrt{\cdot})$ is continuous, their composition $f(\beta \sqrt{a(\cdot)})$ is also Riemann integrable. Similarly, $\sqrt{a(\cdot)}$ is also Riemann integrable. Thus, by taking $N\rightarrow\infty$, \eqref{eq:second sandwich} yields
\begin{align}
\lim_{N\rightarrow\infty}\abs{\frac{1}{N} \sum_{k=1}^{K_N-1} \EX{\log Z^{(kM)}_{\beta,M}} - \int_0^1 f(\beta\sqrt{a(s)})\dd{s}} \leq \varepsilon \beta \int_0^1 \sqrt{a(s)}\dd{s}. \label{eq:eps upp}
\end{align}
Since $\varepsilon>0$ is arbitrary chosen, \eqref{eq:eps upp} implies that
\begin{align*}
\lim_{N\rightarrow\infty}\frac{1}{N} \sum_{k=1}^{K_N-1} \EX{\log Z^{(kM)}_{\beta,M}}
= \int_0^1 f(\beta\sqrt{a(s)})\dd{s}
= \tilde{F}_\beta,
\end{align*}
as desired.
\end{proof}

\section{A property of the Gibbs measure in the supercritical regime} \label{sec:supercritical Gibbs}

From now on, we assume that $A$ is non-concave, so $\beta_G<\infty$. Also, we suppose that $\beta>\beta_G$. The goal of this section is to show that the Gibbs measure tends to sample a vertex that has an ancestor that jumps exceptionally high. The meaning of having an ancestor that jumps exceptionally high is quantified in the following definition.
\begin{definition}
\label{def:bad ancestor}
Given $z>0$, $K\in\N$ and a CREM $\mathbf{X}$, a vertex $v\in\T_N$ with $\abs{v}=n\in\llbracket 1, N\rrbracket$ is said to have a $(z,K,\mathbf{X})$-steep ancestor if there exists $k\in \llbracket 1,\floor{nK/N}\rrbracket$ such that 
\begin{align*}
X_{v[\floor{Nk/K}]} - X_{v[\floor{N(k-1)/K}]} > N \sqrt{2\log 2 (1+z) a_k},
\end{align*}
where $a_k = (A(k/K)-A((k-1)/K))/K$. 
\end{definition}

The goal of this section can now be phrased as the following proposition.
\begin{proposition}
\label{prop:Gibbs sampling}
Let $\beta>\beta_G$. There exist $z>0$, $K\in\N$ such that, for all $\delta>0$ sufficiently small,
\begin{align*}
\lim_{N\rightarrow\infty } \PR{\sum_{\abs{u}=N} \mu_{\beta,N}(u)\Ind\{\text{$u$ has a $(z,K,\mathbf{X})$-steep ancestor\}} > 1-e^{-\delta N} } = 1.
\end{align*}
\end{proposition}
% Key Proposition!

The proof of Proposition~\ref{prop:Gibbs sampling} is based on the following lemma which states the free energy converges to $F_\beta$ in probability.
\begin{lemma}
\label{lem:concentration of Fe}
For all $\beta>0$, for all $\varepsilon>0$, we have 
\begin{align*}
\lim_{N\rightarrow\infty}\PR{\abs{\frac{1}{N}\log Z_{\beta,N} - F_\beta}>\varepsilon} = 0.
\end{align*}
\end{lemma}
\begin{proof}
For all $\beta>0$, for all $\varepsilon>0$, the concentration inequality of free energies (see, Theorem 1.2 in \cite{SK_Panchenko}) states that
\begin{align*}
\PR{\abs{\frac{1}{N}\log Z_{\beta,N} - \frac{1}{N}\EX{\log Z_{\beta,N}}}>\varepsilon} \leq 2 \exp(-\frac{\varepsilon^2}{4\beta^2}N).
\end{align*}
Then, the proof is completed by incorporating Fact~\ref{fact:Fe}.
\end{proof}

We now prove Proposition~\ref{prop:Gibbs sampling}.

\begin{proof}[Proof of Proposition~\ref{prop:Gibbs sampling}]
For all $u\in\partial\T_N$, let $A_u$ be the set where $u$ does not have a $(z,K,\mathbf{X})$-steep ancestor defined as
\begin{align}
A_u\coloneqq \left\{\forall k\in \llbracket 1, K\rrbracket : X_{u[\floor{Nk/K}]}-X_{u[\floor{N(k-1)/K}]} \leq N\sqrt{2\log 2(1+z) a_k}\right\}. \label{eq:def of A_u}
\end{align}
To prove Proposition~\ref{prop:Gibbs sampling}, it suffices to show that there exist $K\in\N$ and $z>0$ such that for all $\delta>0$ sufficiently small,
\begin{align*}
\limsup_{N\rightarrow\infty} \PR{\sum_{\abs{u}=N} \frac{e^{\beta X_u}}{Z_{\beta,N}}\Ind_{A_u} \geq e^{-\delta N} } = 0,
\end{align*}
Because the function $a(\cdot)$ is Riemann integrable and $f(\beta \sqrt{\cdot})$ is continuous, their composition $f(\beta \sqrt{a(\cdot)})$ is also Riemann integrable. On the other hand, since $\beta > \beta_G$, Proposition~\ref{prop:FE comparison} implies thats $F_\beta - \tilde{F}_\beta > 0$. Therefore, we can choose $z>0$ sufficiently small and $K\in\N$ sufficiently large such that
\begin{align}
F_\beta - (1+z)\frac{1}{K}\sum_{k=1}^K \max_{s\in [(k-1)/K,k/K]}f(\beta\sqrt{a(s)}) \geq C > 0, \label{eq:good parameters}
\end{align}
for some $C>0$. In the rest of the proof, we fix our choice of $z$ and $K$. We also fix $\delta>0$ and $c>0$ sufficiently small such that $C-\delta-c>0$.

Now,
\begin{align}
&\PR{\sum_{\abs{u}=N} \frac{e^{\beta X_u}}{Z_{\beta,N}}\Ind_{A_u} \geq e^{-\delta N}} \nonumber \\
&\leq \PR{Z_{\beta,N} < \exp(F_\beta(1-c)N)} \nonumber \\
&+\PR{\Bigg\{\sum_{\abs{u}=N} \frac{e^{\beta X_u}}{Z_{\beta,N}}\Ind_{A_u} \geq e^{-\delta N} \Bigg\} \cap \Bigg\{Z_{\beta,N} \geq \exp(F_\beta (1-c)N)\Bigg\}}
\nonumber \\
&\leq 
\PR{Z_{\beta,N} < \exp(F_{\beta}(1-c)N)}
+\PR{\sum_{\abs{u}=N} e^{\beta X_u}\Ind_{A_u} \geq e^{-\delta N}e^{F_\beta(1-c)N}}. \label{eq:Z^G upper bound}
\end{align}
By Lemma~\ref{lem:concentration of Fe}, the first term in \eqref{eq:Z^G upper bound} tends to $0$ as $N\rightarrow\infty$. Thus, it remains to prove the second probability in \eqref{eq:Z^G upper bound} converges to $0$ as $N\rightarrow\infty$. Let $Y_k \sim N(0,NKa_{k,N})$ and $a_{k,N} \coloneqq (A(\floor{kN/K}/N)-A(\floor{(k-1)N/K}/N))/K$ for all $k\in\llbracket 1,K\rrbracket$. By completing the square, we have
\begin{align}
&\EX{\sum_{\abs{u}=N} e^{\beta X_u}\Ind_{A_u}} \nonumber \\
&\leq 2^N \prod_{k=1}^K e^{\beta^2NKa_{k,N}/2} \PR{Y_k \leq N(\sqrt{2\log 2 (1+z) a_k}-\beta Ka_{k,N})} \nonumber \\
&= \prod_{k=1}^K 2^{N/K} e^{\beta^2NKa_{k,N}/2} \PR{Y_k \leq N(\sqrt{2\log 2 (1+z) a_k}-\beta K a_{k,N})}. \label{eq:expect B uppbnd} 
\end{align}
\paragraph{Case 1.} If $\sqrt{2\log 2 (1+z) a_k} < \beta K a_{k,N}$, the Chernoff bound yields 
\begin{align*}
&2^{N/K} e^{\beta^2N K a_{k,N}/2} \PR{Y_k \leq N(\sqrt{2\log 2 (1+z) a_k}-\beta K a_{k,N}} \\
&\leq 2^{N/K} e^{\beta^2N K a_{k,N}/2} \exp(-\frac{N^2(\sqrt{2\log 2 (1+z) a_k}-\beta K a_{k,N})^2}{2NKa_{k,N}}) \\
&= 2^{N/K} e^{\beta^2N K a_{k,N}/2} \exp(-N\log 2(1+z)a_k/Ka_{k,N}) \\
&\quad\quad\quad\quad\quad\quad\quad\quad\quad\quad\quad\quad\quad 
\cdot \exp(\beta N \sqrt{2\log 2 (1+z) a_k})\exp(-N\beta^2 K a_{k,N}/2) \\
&= \underbrace{\exp(-N(\log 2)za_k/Ka_{k,N})}_{\leq 1}\exp(N \log 2 (1 - a_k/a_{k,N})/K)\exp(N\beta\sqrt{2\log 2 (1+z) a_k}) \\
&\leq \exp(N \log 2 (1 - a_k/a_{k,N})/K) \exp(N\beta\sqrt{2\log 2 (1+z) a_k}).
\end{align*}
\paragraph{Case 2.} If $\sqrt{2\log 2 (1+z) a_k} \geq \beta K a_{k,N}$, we simply bound the probability in \eqref{eq:expect B uppbnd} by 1 and obtain
\begin{align*}
&2^{N/K} e^{\beta^2N K a_{k,N}/2} \PR{Y_k \leq N(\sqrt{2\log 2 (1+z) a_k}-\beta K a_k)} \\
&\leq 2^{N/K} e^{\beta^2N K a_{k,N}/2}
\end{align*}
Then, by \eqref{eq:expect B uppbnd} and the two cases above, we have
\begin{align}
&\limsup_{N\rightarrow\infty}\frac{1}{N}\log \EX{\sum_{\abs{u}=N} e^{\beta X_u}\Ind_{A_u}} \nonumber \\
&\leq \sum_{\sqrt{2\log 2 (1+z) a_k} < \beta K a_{k}} \beta\sqrt{2\log 2 (1+z) a_k}
+ \sum_{\sqrt{2\log 2 (1+z) a_k} \geq \beta K a_{k}} \frac{\log 2}{K} +  \frac{\beta^2 K a_k}{2} \nonumber \\
&\leq (1+z)\frac{1}{K}\sum_{k=1}^K f(\beta/(1+z) \sqrt{a_k}K) \nonumber \\
&\leq (1+z)\frac{1}{K}\sum_{k=1}^K \max_{s\in [(k-1)/K,k/K]}f(\beta\sqrt{a(s)}), \label{eq:riemann_sum}
\end{align}
where \eqref{eq:riemann_sum} follows from monotonicity of the function $f$.
By the Markov inequality, \eqref{eq:good parameters}, the second term in \eqref{eq:Z^G upper bound} satisfies the following
\begin{align}
\limsup_{N\rightarrow\infty}\frac{1}{N}\log\PR{\sum_{\abs{u}=N} e^{\beta X_u}\Ind_{A_u} \geq e^{-\delta N}e^{F_\beta(1-c)N}}
\leq -C +c+\delta < 0. \label{eq:negative exponent}
\end{align}
where $C$, $c$ and $\delta$ are chosen as in the first paragraph of the proof. Combining \eqref{eq:negative exponent} and Proposition~\ref{prop:FE comparison}, we conclude that
\begin{align*}
\limsup_{N\rightarrow\infty} \PR{\sum_{\abs{u}=N} e^{\beta X_u}\Ind_{A_u} \geq e^{-\delta N}e^{F_\beta(1-c)N}} = 0,
\end{align*}
and the proof is completed.
\end{proof}

\section{Hardness in the supercritical regime} \label{sec:hardness}

Assume that $A$ is non-concave and $\beta>\beta_G$. The goal of this section is to prove Theorem~\ref{main:hardness}. Before we dive in the section, we introduce a few definitions. The first is a chain of subtrees defined as follows.
\begin{definition}\label{def:chain of subtrees}
For $v\in\T$, let $\mathcal{C}_v$ be a chain of subtrees containing $v$ and all its ancestors defined by
\begin{align*}
\mathcal{C}_v = \bigcup_{k=0}^{\floor{N\abs{v}/K}} \T_{\floor{N(k+1)/K}-\floor{Nk/K}}^{v[\floor{Nk/K}]}.
\end{align*} 
\end{definition}

\begin{remark}
See Figure~\ref{fig:chain of trees} for an illustration of Definition~\ref{def:chain of subtrees}. Also, note that in particular, $v\in \mathcal C_v$ for every $v\in \T_N$.
\end{remark}

% Figure of the chain of trees
\begin{figure}[ht]
\centering
\begin{tikzpicture}[scale=1.2]

% Draw chain of isosceles triangles (chain of subtrees)
\draw [gray!80, fill=gray!30] (0,0) -- (-0.75,-0.5) -- (0.75,-0.5) -- cycle;
\draw [gray!80, fill=gray!30] (-0.4,-0.5) -- (-0.4-0.75,-1) -- (-0.4+0.75,-1) -- cycle;
\draw [gray!80, fill=gray!30] (-0.2,-1) -- (-0.2-0.75,-1.5) -- (-0.2+0.75,-1.5) -- cycle;

% Draw the last isosceles triangle (the last subtree in the chain)
\draw [gray!80, fill=gray!30] (2,-2.75) -- (2-0.75,-3.25) -- (2+0.75,-3.25) -- cycle;

% Draw large isosceles triangle (the large tree)
\draw (0,0) -- (-6,-4) -- (6,-4) -- cycle;
%\draw (0,0) -- (2,-0.5) -- (2.5,-1) -- (-2.5,-1) -- (-2,-0.5) -- cycle;

% Draw the wobbly path from v to root
\draw [red, thick, rounded corners=5pt] (0,0) -- (-0.1,-0.2) -- (0.2,-0.3) -- (-0.4,-0.5);
\draw [red, thick, rounded corners=5pt] (-0.4,-0.5) -- (-0.6,-0.7) -- (-0.1,-0.9) -- (-0.2,-1);
\draw [red, thick, rounded corners=5pt] (-0.2,-1) -- (-0.1,-1.2) --  (0.2,-1.3) -- (-0.1,-1.5);
\draw [red, thick, rounded corners=5pt] (-0.1,-1.5) -- (-0.2,-1.8) -- (0.2,-2) -- (1.3,-2.3) -- (2,-2.75) ;
%-- (1.9,-2.9) -- (2,-3);
\draw [red, thick, rounded corners=5pt] (2,-2.75) -- (2.22,-3.05) -- (2.2,-3.15);

% Draw the roots of the chain of subtrees
\fill[red] (0,0) circle (0.04) node[above, text=black] {$\varnothing$}; % origin
\fill[red] (-0.4,-0.5) circle (0.04);
\fill[red] (-0.2,-1) circle (0.04);
%\fill[red] (-0.1,-1.5) circle (0.04);
\fill[red] (2,-2.75) circle (0.04);
\fill (2.2,-3.15) circle (0.04) node[right] {$v$}; %v
\end{tikzpicture}
\caption{A schematic illustration of the set $\mathcal{C}_v$ appearing in Definition~\ref{def:chain of subtrees}. The largest isosceles triangle represents the binary tree $\T_N$. The chain of isosceles triangles colored in gray represent the set $\mathcal{C}_v$, where the $k$-th subtree has depth $\floor{N(k+1)/K}-\floor{Nk/K}$. The red dots represent $v[\floor{Nk/K}]$, the ancestors of $v$ at depth $\floor{Nk/K}$.}
\label{fig:chain of trees}
\end{figure}
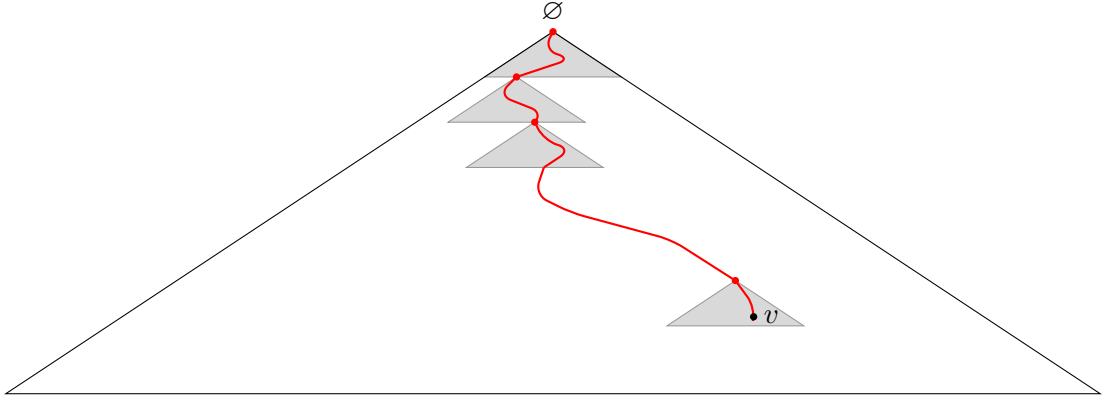

Next, we introduce the following stopping time.

\begin{definition}\label{def:tau'}
Let $(v(n))_{n\in\N}$ be an algorithm. Define the stopping time $\tau'$ as the first time $n$ when the algorithm finds a vertex in $\mathcal{C}_{v(n)}$ with a $(z,K,\mathbf{X})$-steep ancestor, given by:
\begin{align*}
\tau' = \inf\left\{n\in\N : \text{$\exists w\in \mathcal{C}_{v(n)}$ such that $w$ has a $(z,K,\mathbf{X})$-steep ancestor}\right\}.
\end{align*}
\end{definition}

Now, we come back to the proof of Theorem~\ref{main:hardness}. The proof is based on the following two propositions. The first proposition asserts that the running time of an algorithm that approximates the Gibbs measure dominates the stopping time $\tau'$ with probability approaching $1$.

\begin{proposition}
\label{prop:tau dominates tau'}
Suppose $A$ to be non-concave. Let $\beta>\beta_G$. If $\tau$ is the running time of an algorithm that approximates the Gibbs measure, then 
\begin{align*}
\lim_{N\rightarrow\infty} \PR{\tau\geq \tau'} = 1.
\end{align*}
\end{proposition}

The proof of Proposition~\ref{prop:tau dominates tau'} is provided in Section~\ref{sec:prop:tau dominates tau'}. Note that Addario-Berry and Maillard proved in \cite{AB&M20} a hardness result of finding a vertex $v\in\partial\T_N$ such that $X_v$ lies in a level set above a critical level, denoted by $x_*N$. In their case, $\tau\geq \tau'$ holds deterministically because they showed in Lemma 3.1 in their paper that any for any vertex $v\in\partial\T_N$ such that $X_v$ lies in a level set above $xN$, where $x>x_*$, $v$ must have a $(z,K,\mathbf{X})$-ancestor. Nevertheless, Proposition~\ref{prop:tau dominates tau'} is sufficient for our purpose.

The second proposition to prove Theorem~\ref{main:hardness} is the following. The proposition asserts that the $\tau'$ is exponentially large with probability approaching $1$.
\begin{proposition} \label{prop:sub problem is hard}
There exists $\gamma>0$ such that
\begin{align*}
\lim_{N\rightarrow\infty}\PR{\tau' > e^{\gamma N}} = 1.
\end{align*}
\end{proposition}
Proposition~\ref{prop:sub problem is hard} is proven following the same argument as in  \cite{AB&M20}, and the proof is included in Section~\ref{sec:sub problem is hard} for completeness.

Conditioned on Proposition~\ref{prop:tau dominates tau'} and Proposition~\ref{prop:sub problem is hard}, the proof of Theorem~\ref{main:hardness} is fairly short, so we provide it here.

\begin{proof}[Proof of Theorem~\ref{main:hardness}]
Fix $\beta>\beta_G$. Let $(v(n))_{n\in\N}$ be an algorithm that approximates the Gibbs measure with probability approaching $1$. Suppose that $\tau$ is its running time and $\tilde{\mu}_{N}$ is its output law, which is the law of $v(\tau)$ conditioned on the CREM.

Now, combining Proposition~\ref{prop:tau dominates tau'} with Proposition~\ref{prop:sub problem is hard}, we conclude that there exists $\gamma>0$ such that
\begin{align*}
\lim_{N\rightarrow\infty}\PR{\tau\geq e^{\gamma N}} \geq \lim_{N\rightarrow\infty} \PR{\{\tau\geq \tau'\}\cap \{\tau'\geq e^{\gamma N}\}} = 1,
\end{align*}
and the proof is completed.
\end{proof}

%but with an additional ingredient which is the following proposition.
%
%, but with an additional element (Proposition \ref{prop:tau dominates tau'}) that adapts their argument to our context.
%
%The strategy goes as follows. A vertex is called bad if, with probability approaching $1$, it takes exponential steps for the given algorithm to sample this vertex. We first show in the following  Then, Theorem~\ref{main:hardness} is proven by an argument adapted from the proof for the second part of Theorem 1.1 in \cite{AB&M20}.

\subsection{Proof of Proposition~\ref{prop:tau dominates tau'}} \label{sec:prop:tau dominates tau'}

The proof of Proposition~\ref{prop:tau dominates tau'} follows from the following lemma which states if an algorithm approximates the Gibbs measure, with probability approaching $1$, its output law also tends to sample a vertex with a $(z,K,\mathbf{X})$-steep ancestor. 

\begin{lemma}
\label{lem:approx Gibbs sampling}
Let $\beta>\beta_G$. Suppose that $\tilde{\mu}_N$ is the output law of an algorithm that approximates the Gibbs measure. Then, there exist $z>0$, $K\in\N$ such that there exists $\varepsilon_N\rightarrow 0$ such that, with probability approaching $1$,
\begin{align*}
\lim_{N\rightarrow\infty} \PR{\sum_{\abs{u}=N} \tilde{\mu}_{N}(u)\Ind\{\text{$u$ has a $(z,K,\mathbf{X})$-steep ancestor\}} > 1-\varepsilon_N} = 1,
\end{align*}
\end{lemma}

We now proceed to the proof of Proposition~\ref{prop:tau dominates tau'}.

\begin{proof}[Proof of Proposition~\ref{prop:tau dominates tau'}]
Fix $\beta>\beta_G$. Let $(v(n))_{n\in\N}$ be an algorithm that approximates the Gibbs measure with probability approaching $1$. Suppose that $\tau$ is its running time and $\tilde{\mu}_{N}$ is its output law, which is the law of $v(\tau)$ conditioned on the CREM.
Recall that $\tau'$ defined in Definition~\ref{def:tau'} is the first time where the algorithm finds a vertex with a $(z,K,\mathbf{X})$-steep ancestor. Therefore, the output $v(\tau)$ has a $(z,K,\mathbf{X})$-steep ancestor implies that $\tau\geq \tau'$. 
Defining $\mathcal{G}_N$ the event 
\begin{align*}
\mathcal{G}_N \coloneqq \left\{\sum_{\abs{u}=N}\tilde{\mu}_N(u) \Ind\{\text{$u$ has a $(z,K,\mathbf{X})$-steep ancestor}\} > 1-\varepsilon_N\right\},
\end{align*}
we have
\begin{align*}
\PR{\tau\geq \tau'}
&\geq \PR{\text{$v(\tau)$ has a $(z,K,\mathbf{X})$-steep ancestor}} \\
&= \EX{\sum_{\abs{u}=N} \tilde{\mu}_N(u)\Ind\{\text{$u$ has a $(z,K,\mathbf{X})$-steep ancestor}\}} && \text{(Definition of $\tilde{\mu}_N$)} \\
&\geq \EX{\Ind_{\mathcal{G}_N}\sum_{\abs{u}=N} \tilde{\mu}_N(u)\Ind\{\text{$u$ has a $(z,K,\mathbf{X})$-steep ancestor}\}} \\
&\geq \PR{\mathcal{G}_N} (1-\varepsilon_N) \rightarrow 1, \quad N\rightarrow\infty, && \text{(By Lemma~\ref{lem:approx Gibbs sampling})}
\end{align*}
and the proof is completed.
\end{proof}

%The proof of Proposition~\ref{lem:approx Gibbs sampling} is contained in Section~\ref{sec:proof:Gibbs sampling}. Finally, we justify in following proposition the reason why a vertex having a $(z,K,\mathbf{X})$-steep ancestor is bad. 

\subsection{Proof of Lemma~\ref{lem:approx Gibbs sampling}} \label{sec:proof of approx Gibbs sampling}
This section is devoted to the proof of Proposition~\ref{lem:approx Gibbs sampling}. The proof relies on Proposition~\ref{prop:Gibbs sampling} and the following lemma which states that if the KL divergence between two sequences of random probability measures are close to each other with probability approaching $1$, and if the measures of certain events in the second sequence decay exponentially to $0$ with probability approaching $1$, then the measures of the corresponding events in the first sequence also converge to $0$ with probability approaching $1$.

\begin{lemma}
\label{lem:approx}
Suppose $(P_N)_{N\in\N}$ and $(Q_N)_{N\in\N}$ be two sequences of \emph{random} probability measures defined on a discrete space $S$ such that the sequence $(P_N)_{N\in\N}$ approximates the sequence $(Q_N)_{N\in\N}$ with probability approaching $1$. If $(A_N)_{N\in\N}$ is a sequence of events on $S$ such that with probability approaching $1$, $Q_N(A_N)$ converges to $0$ exponentially fast as $N\rightarrow\infty$, i.e., there exists $c>0$ such that 
\begin{align*}
\lim_{N\rightarrow\infty}\PR{Q_N(A_N) \leq e^{-cN}}=1,
\end{align*}
then $P_N(A_N)$ converges to $0$ with probability approaching $1$ as $N\rightarrow\infty$, i.e., there exists $\varepsilon_N\rightarrow 0$ such that
\begin{align*}
\lim_{N\rightarrow\infty} \PR{P_N(A_N) \leq \varepsilon_N} = 1.
\end{align*}
\end{lemma}
Lemma~\ref{lem:approx} follows from the so-called Birg\'{e}'s inequality which, roughly speaking, says that if for two probability measures $P$ and $Q$ defined on the same probability space such that $P$ is dominated by $Q$, for any event $A$, the difference of between $P(A)$ and $Q(A)$ is gauged by the KL divergence from $P$ to $Q$.

\begin{fact}[Birg\'{e}'s inequality, Theorem~4.20 in \cite{ConIneq13}]\label{Birge}
Let $P$ and $Q$ be two probability measures defined on probability space $(S,\mathcal{S})$ such that $P$ is dominated by $Q$, i.e., for all event $A\in\mathcal{S}$, $Q(A)=0$ implies $P(A)=0$. Then,
\begin{align*}
\sup_{A\in\mathcal{A}} h(P(A),Q(A)) \leq \kld{P}{Q},
\end{align*}
where $h(p,q)=p\log(p/q)+(1-p)\log((1-p)/(1-q))$ is the relative entropy between two Bernoulli distribution with parameters $p$ and $q$, respectively.
\end{fact}

\begin{remark}
Note that the positions of $P$ and $Q$ are swapped comparing with the statement of Theorem~4.20 in \cite{ConIneq13}.
\end{remark}

Next, we state a simple but handy fact of the function $x\mapsto x\log x$ where the proof is omitted.
%We provide its proof for completeness.

\begin{fact}\label{fact:xlogx}
The range of the function $g$ defined by $g(0)=0$ and $g(x)=x\log x$ on $(0,1]$ equals $[-e^{-1},0]$. 
\end{fact}
%\begin{proof}
%Note that the function $g$ is continuous on $[0,1]$ since it is a multiplication of two continuous function on $(0,1]$ and $g$ is continuous at $0$ by L'H\^{o}pital's rule. If we show that $-e^{-1}$ and $0$ are the global minimum and global maximum, respectively, of the function $g$, the intermediate value theorem implies that the range of $g$ equals $[-e^{-1},0]$. Thus, it remains to determine the global minimum/maximum of $g$. Moreover, since $g$ is differentiable on $(0,1)$, the global extreme values only occurs at the boundary or at the critical points. One can check that $e^{-1}$ is the only critical point of $g$. Finally, since $g(0)=g(1)=0$ and $g(e^{-1})=-e^{-1}$, the proof is completed. 
%\end{proof}

We are now ready to prove Lemma~\ref{lem:approx}.

\begin{proof}[Proof of Lemma~\ref{lem:approx}]
Let $(P_N)_{N\in\N}$ and $(Q_N)_{N\in\N}$ be two sequences of \emph{random}  probability measures defined on a discrete space $S$. Suppose that the sequence $(P_N)_{N\in\N}$ approximates the sequence $(Q_N)_{N\in\N}$ with probability approaching $1$, and there exists a sequence of event $(A_N)_{N\in\N}$ such that there exists $c>0$ such that for all $N\in\N$,
\begin{align}
\PR{Q_N(A_N) \leq e^{-cN}}=1-o_N(1). \label{eq:expo decay}
\end{align}
For all $N\in\N$, on the event $\{Q_N(A_N) \leq e^{-cN}\}$, Fact~\ref{fact:xlogx} yields
\begin{align}
h(P_N(A_N),Q_N(A_N)) 
&= P_N(A_N)(\log P_N(A_N) - \log Q_N(A_N)) \nonumber \\
&+ (1-P_N(A_N))(\log (1-P_N(A_N)) - \underbrace{\log (1-Q_N(A_N))}_{\leq 0}) \nonumber \\
&\geq P_N(A_N)cN - 2e^{-1}. \label{eq:entropy lowerbound} 
\end{align}
Therefore, combining \eqref{eq:entropy lowerbound} and Birg\'{e}'s inequality, we have
\begin{align}
P_N(A_N) 
\leq \frac{1}{cN} h(P_N(A_N),Q_N(A_N)) + \frac{1}{cN}2e^{-1}
\leq \frac{1}{cN} \kld{P_N}{Q_N} + \frac{1}{cN}2e^{-1}.
\label{eq:P(A) upper bound}
\end{align}
On the other hand, since $P_N$ approximates $Q_N$ with probability approaching $1$, there exists $\varepsilon_N\rightarrow 0$ such that
\begin{align*}
\PR{\frac{1}{N}\kld{P_N}{Q_N} \leq \varepsilon_N} = 1 - o_N(1).
\end{align*}
Thus, by \eqref{eq:P(A) upper bound}, with probability approaching $1$, 
\begin{align*}
P_N(A_N) \leq \frac{\varepsilon_N}{c} + \frac{1}{cN} 2e^{-1} \rightarrow 0, \quad N\rightarrow\infty
\end{align*}
as desired.
\end{proof}

We now prove Proposition~\ref{lem:approx Gibbs sampling}.

\begin{proof}[Proof of Proposition~\ref{lem:approx Gibbs sampling}]
Fix $\beta > \beta_G$. Let $\tilde{\mu}_N$ be the output law of an algorithm that approximates the Gibbs measure. We apply Lemma~\ref{lem:approx} with $P_N\coloneqq \mu_{\beta,N}$, $Q_N\coloneqq \tilde{\mu}_N$ and $A_N$ defined in \eqref{eq:def of A_u} where its complement equals
\[A_N^c \coloneqq \left\{u\in\partial\T_N : \text{$u$ has a $(z,K,\mathbf{X})$-steep ancestor}\right\}.\] 
Since Proposition~\ref{prop:Gibbs sampling} implies that there exists $\delta>0$ such that \begin{align*}
\lim_{N\rightarrow\infty} \PR{P_N(A_N^c)\geq 1-e^{-\delta N}} = 1,
\end{align*}
we then conclude from Lemma~\ref{lem:approx} that there exists $\varepsilon_N\rightarrow 0$ such that
\begin{align*}
&\lim_{N\rightarrow\infty} \PR{\sum_{\abs{u}=N} \tilde{\mu}_{N}(u)\Ind\{\text{$u$ has a $(z,K,\mathbf{X})$-steep ancestor\}} > 1-\varepsilon_N} \\
&=\lim_{N\rightarrow\infty} \PR{Q_N(A_N^c)\geq 1-\varepsilon_N} = 1,
\end{align*}
and proof is completed.
\end{proof}

\subsection{Proof of Proposition~\ref{prop:sub problem is hard}} \label{sec:sub problem is hard}

This section is devoted to the proof of Proposition~\ref{prop:sub problem is hard}. As the proof is modified from the proof for the second part of Theorem 1.1 in \cite{AB&M20}, we start by recalling some relevant notation and lemmas from that article.

\paragraph{Notation.} 
%Define for any $v\in\T_N$ the following set of vertices:
%\begin{align}
%\mathcal C_v &= \bigcup_{k=1}^K C_{v,k}, \quad \text{where} \nonumber \\
%C_{v,k} &= \{w\in\T_N:\text{$\abs{w}\leq \floor{Nk/K}$ and $\floor{N(k-1)/K}\leq \abs{w\wedge v}\leq \floor{Nk/K}$}\}. \nonumber
%\end{align}
For $v\in\T_N$, recall the definition of $\mathcal{C}_v$ in Definition~\ref{def:chain of subtrees}.
We then define the filtration
\[
\G_k = \sigma\left(v(1),\ldots,v(k);\,(X_{w})_{w\in\mathcal C_{v(1)}},\ldots,(X_{w})_{w\in\mathcal C_{v(k)}};\,U_1,\ldots,U_{k+1}\right).
\]
Note that $\tilde{\F}_k \subset \G_k$ for all $k\ge0$ --- heuristically, $\G_k$ adds to $\tilde{\F}_k$ the information about the values in the branching random walk of all vertices contained in $\mathcal C_{v(i)}$, $i=1,\ldots,k$. Note that trivially, the stochastic process $v(n)_{n\ge0}$ is still measurable with respect to this larger filtration $\G$.
For $n \ge 1$, let $\mathcal{R}_n = \bigcup_{i=1}^n \mathcal{C}_{v(i)}$ be the union of $\mathcal C_{v(i)}$, $i=1,\ldots,k$. 
Also, let $\hat{v}(n)$ be the most recent ancestor of $v(n)$ in $\mathcal R_{n-1}$ if $n > 1$, and let $\hat{v}(n)$ be the root of $\mathbb{T}_n$ if $n=1$.
Finally, $\mathbf{X}'\coloneqq (X'_v)_{v \in \T_n}$ is a i.i.d. copy of $\mathbf{X}=(X_v)_{v \in \T_n}$ and is independent of $\G_{n-1}$.

Now, we recall the statements of two lemmas in \cite{AB&M20} which will be useful in the proof of Proposition~\ref{prop:sub problem is hard}. The first lemma is a direct implication of the branching property.

\begin{lemma}[Lemma 3.2 in \cite{AB&M20}]
\label{lem:independence}
Fix any randomized search algorithm $\mathrm{v}=(v(n))_{n \ge 1}$. Then conditioned on $\G_{n-1}$, the family of random variables $(X_v-X_{\hat{v}(n)})_{v\in \mathcal{R}_n\setminus \mathcal{R}_{n-1} }$ has the same law as $(X'_v-X'_{\hat{v}(n)})_{v\in \mathcal{R}_n\setminus \mathcal{R}_{n-1}}$. 
\end{lemma}

The next lemma states, roughly speaking, that $(z,K,\mathbf{X})$-steep vertices are rare.

\begin{lemma}[Lemma 3.3 in \cite{AB&M20}]
\label{lem:subproblem}
For all $K\in\N$ and $z>0$, for any $\gamma \in (0,(z \log 2)/K)$, for all $N$ sufficiently large, for any $w\in\T_N$, 
\[
\PR{\exists~v\in\mathcal C_{w}~:~v\mbox{ is }\mbox{$(z,K,\mathbf{X})$-steep}} \le e^{-\gamma N}
\] 
\end{lemma}

We now proceed to the proof of Proposition~\ref{prop:sub problem is hard}.

\begin{proof}[Proof of Proposition~\ref{prop:sub problem is hard}]
The goal is to show that $\tau'$  stochastically dominates a geometric random variable with an exponentially small parameter, which follows from an argument slightly adapted from the proof for the second part of Theorem 1.1 in \cite{AB&M20}. The argument goes as follows.

By Lemma~\ref{lem:independence}, 
\begin{align}
& \cPR{\exists v\in\mathcal{R}_n\setminus \mathcal{R}_{n-1} : \text{$v$ is $(z,K,\mathbf{X})$-steep}}{\G_{n-1}} \nonumber \\
 = \ &\cPR{\exists v\in\mathcal{R}_n\setminus \mathcal{R}_{n-1} : \text{$v$ is $(z,K,\mathbf{X}')$-steep}}{\G_{n-1}} \nonumber \\
 \le \ &
\cPR{\exists v\in\mathcal{C}_{v(n)} : \text{$v$ is $(z,K,\mathbf{X}')$-steep}}{\G_{n-1}} \nonumber \\
\le \ &
\sup_{w\in\T_N} 
\PR{\exists v\in\mathcal{C}_{w} : \text{$v$ is $(z,K,\mathbf{X}')$-steep}}. \label{eq:steep sup}
\end{align}
The first inequality uses the fact that $\mathcal{R}_n\setminus \mathcal{R}_{n-1} \subset \mathcal{C}_{v(n)}$ and the second inequality uses the independence of $\mathbf{X}'$ and $\G_{n-1}$.
Since $\mathbf{X}'$ and $\mathbf{X}$ have the same law, 
by Lemma~\ref{lem:subproblem}, with $\gamma=\gamma(K,z)$ as in that lemma, \eqref{eq:steep sup} yields that 
\[
\cPR{\exists v\in\mathcal{R}_n\setminus \mathcal{R}_{n-1} : \text{$v$ is $(z,K,\mathbf{X})$-steep}}{\G_{n-1}} \le  e^{-\gamma N}.
\]
We thus obtain 
\begin{align*}
\PR{\tau'=n}
& = \EX{\Ind_{\tau > n-1}\cdot \cPR{\tau=n}{\G_{n-1}}} \\
& = \EX{\Ind_{\tau > n-1}\cdot  \cPR{\exists v\in\mathcal{R}_n\setminus \mathcal{R}_{n-1} : \text{$v$ is $(z,K,\mathbf{X})$-steep}}{\G_{n-1}}} \\
& \le \PR{\tau > n-1} \cdot e^{-\gamma N} \, ,
\end{align*}
from which we conclude that $\tau'$ stochastically dominates a geometric random variable with success probability $e^{-\gamma N}$. In particular, this implies that for any positive constant $\gamma' \in (0,\gamma)$, 
\begin{align*}
\PR{\tau' \geq e^{\gamma'N}} \geq \sum_{n=\ceil{e^{\gamma'N}}}^\infty e^{-\gamma N} (1-e^{-\gamma N})^n \sim \exp(-e^{-(\gamma-\gamma')N})\rightarrow 1,
\end{align*}
as $N\rightarrow\infty$.
\end{proof}

\appendix

\section{A lower bound of the free energy \texorpdfstring{$F_\beta$}{}} \label{ann:lowbnd of Fe}

Recall that the free energy of the CREM is defined in \eqref{def:Fe} as 
\begin{align*}
F_\beta
&\coloneqq
\lim_{N\rightarrow\infty}\frac{1}{N} \EX{\log Z_{\beta,N}}.
\end{align*}
Recall also that we defined in \eqref{eq:tilde F_beta} the quantity 
\begin{align*}
\tilde{F}_{\beta} = \int_0^1 f(\beta\sqrt{a(s)}) \dd{s},
\end{align*}
where
\begin{align*}
f(\beta)
=
\begin{cases}
\displaystyle
\log 2 + \frac{\beta^2}{2}, & \beta < \sqrt{2\log 2} \\
\\
\sqrt{2\log 2} \beta, & \beta \geq \sqrt{2\log 2}.
\end{cases}
\end{align*}
The main goal of this section is to prove the following proposition, which asserts that $F_\beta\geq \tilde{F}_{\beta}$, and that equality holds if and only if $\beta\leq \beta_G$.
\begin{proposition}
\label{prop:FE comparison}
Suppose that $A$ is non-concave. For all $\beta\in [0,\infty)$, define 
\begin{align*}
G_\beta \coloneqq F_\beta - \tilde{F}_\beta.
\end{align*}
Then,
\begin{enumerate}[(i)]
\item \label{prop:FE comparison.1} For all $\beta \in [0,\beta_G]$, $G_\beta=0$.
\item \label{prop:FE comparison.2} For all $\beta > \beta_G$, $G_\beta'>0$. In particular, this implies that $G_\beta>0$ for all $\beta>\beta_G$.
\end{enumerate}
\end{proposition}

Before starting the proof, we recall that the free energy of the CREM has the following formula which can be found in Bovier and Kurkova \cite{CREM04} based on previous results by Capocaccia et al. \cite{GREMFreeEnergy}.
\begin{fact}
\label{fact:Fe}
Given $\beta>0$, let $t_0(\beta) = \sup\{t\in [0,1]:\hat{a}(t)>2\log 2/\beta^2\}$. Then, the free energy of the CREM is given as follows
\begin{align}
F_{\beta}
&= \beta \sqrt{2\log 2}\int_0^{t_0(\beta)} \sqrt{\hat{a}(t)} \dd{t}
+ \frac{\beta^2}{2} (1-\hat{A}(t_0(\beta))) + \log 2 (1-t_0(\beta)) \label{eq:Fe.CREM04} \\
&= \int_0^1 f(\beta\sqrt{\hat{a}(s)})\dd{s}, \label{eq:Fe.speed changed}
\end{align}
where $f(\beta)$ is defined as \eqref{eq:BRW Fe}.
\end{fact}
\begin{proof}
The proof of \eqref{eq:Fe.CREM04} can be found in Theorem 3.3 of Bovier and Kurkova \cite{CREM04}. While the authors of that paper assumed that the function $A$ has to be continuously differentiable, their result can be extended to the case where $A$ is merely Riemann integrable. This extension is possible because their argument is based on the following two ingredients. The first ingredient is the free energy formula for the GREM, given by Capocaccia et al. \cite{GREMFreeEnergy}. As a remark, the definition of the GREM is identical to the definition of the CREM except that the function $A$ is a step function. The second ingredient is a Gaussian comparison argument that only requires the Riemann integrability of $a$.

Now, by \eqref{eq:Fe.CREM04} and the fact that $\hat{a}$ is non-increasing, 
\begin{align*}
&\beta \sqrt{2\log 2}\int_0^{t_0(\beta)} \sqrt{\hat{a}(t)} \dd{t}
+ \frac{\beta^2}{2} (1-\hat{A}(t_0(\beta))) + \log 2 (1-t_0(\beta)) \nonumber \\
&= \beta \sqrt{2\log 2}\int_0^{t_0(\beta)} \sqrt{\hat{a}(t)} \dd{t}
+ \int_{t_0(\beta)}^1  \left(\frac{\beta^2}{2}\hat{a}(t) + \log 2\right) \dd{t} \\
&= \int_0^{t_0(\beta)} f(\beta\sqrt{\hat{a}(t)})\dd{t} + \int_{t_0(\beta)}^1 f(\beta\sqrt{\hat{a}(t)})\dd{t} = \int_0^1 f(\beta\sqrt{\hat{a}(t)})\dd{t},
\end{align*}
which proves \eqref{eq:Fe.speed changed}.
\end{proof}

To prove Proposition~\ref{prop:FE comparison}, we require the following three lemmas. The first lemma provides some useful properties of the function $A$ and its concave hull $\hat{A}$. 

\begin{lemma}
\label{lem:prop of concave hull}
The following are true.
\begin{enumerate}[(i)]
\item \label{lem:prop of concave hull.1} On the set $\{A=\hat{A}\}$, $a=\hat{a}$ almost everywhere. 
\item \label{lem:prop of concave hull.2} Suppose that $A$ is non-concave. Let $\mathcal{I}$ be a connected component of $\{t\in [0,1]:A(t)<\hat{A}(t)\}$. Then, $\hat{a}$ is equal to a positive constant on the interior of $\mathcal{I}$, denoted by $\hat{a}_\mathcal{I}$. Moreover, 
\begin{align*}
\int_\mathcal{I} a(s) \dd{s} = \int_\mathcal{I} \hat{a}(s) \dd{s} = \hat{a}_\mathcal{I}\abs{\mathcal{I}},
\end{align*}
where $\abs{\mathcal{I}}$ denotes the Lebesgue measure of $\mathcal{I}$.
\item \label{lem:prop of concave hull.3} With the same assumptions as in \eqref{lem:prop of concave hull.2}, we have
\begin{align*}
\int_\mathcal{I} \sqrt{a(s)} \dd{s} < \int_\mathcal{I} \sqrt{\hat{a}(s)} \dd{s} = \sqrt{\hat{a}_\mathcal{I}}\abs{\mathcal{I}}.
\end{align*}
\end{enumerate}
\end{lemma}
\begin{proof}
We prove this lemma by addressing each point separately.

\noindent Proof of \eqref{lem:prop of concave hull.1}. The set $\{A=\hat{A}\}$ is Lebesgue measurable because $\{A=\hat{A}\}=(\hat{A}-A)^{-1}(\{ 0\})$ and the function $\hat{A}-A$ is continuous. If $\{A=\hat{A}\}$ is of measure zero, the statement trivially holds

Suppose now that $\{A=\hat{A}\}$ has positive measure. Note that $\{A=\hat{A}\}$ contains all the global maximum points of $\hat{A}-A$ as $\hat{A}\geq A$.  Thus, by Fermat's theorem of stationary points, for all $t\in \{A=\hat{A}\}\cap \{\text{$\hat{A}-A$ is differentiable}\}$, we have $\hat{a}(t)=a(t)$. It remains to show that $t\in \{A=\hat{A}\}\cap \{\text{$\hat{A}-A$ is not differentiable}\}$ is of measure zero. By the fundamental theorem of calculus and the fact that $\hat{a}-a$ is continuous almost everywhere on $[0,1]$, the function $\hat{A}-A$ is differentiable almost everywhere on $[0,1]$. Therefore, $\{A=\hat{A}\}\cap \{\text{$\hat{A}-A$ is not differentiable}\}$ is of measure zero.

\noindent Proof of \eqref{lem:prop of concave hull.2}. Let $\mathcal{I}$ be a connected component of $\{A\neq \hat{A}\}$ with endpoints $t_1$ and $t_2$. By the continuity of $A$ and $\hat{A}$, $A(t_1)=\hat{A}(t_1)$ and $A(t_2)=\hat{A}(t_2)$. By the minimality of $\hat{A}$, for all $t\in\text{int}(\mathcal{I})$, $\hat{A}(t)$ is equals to the linear interpolation between $A(t_1)$ and $A(t_2)$. In particular, this implies that the $\hat{a}$ is constant on $\text{int}(\mathcal{I})$. Moreover, $\hat{a}$ has to be positive. Otherwise, by the fundamental theorem of calculus, $A(t)=\hat{A}(t)$ for any $t\in\text{int}(\mathcal{I})$ which contradicts the assumption that $\mathcal{I}$ is a connected component of $\{A\neq \hat{A}\}$.

To prove the second statement of \eqref{lem:prop of concave hull.2}, note that 
\begin{align*}
\int_\mathcal{I} a(s) \dd{s} = A(t_2) - A(t_1) = \hat{A}(t_2) - \hat{A}(t_2) = \int_\mathcal{I} \hat{a}(s) \dd{s} = \hat{a}_\mathcal{I}\abs{\mathcal{I}}.
\end{align*}

\noindent Proof of \eqref{lem:prop of concave hull.3}. By \eqref{lem:prop of concave hull.2}, $\hat{a}_\mathcal{I}$ is positive. Then, by the Cauchy--Schwarz inequality, 
\begin{align}
\int_\mathcal{I} \sqrt{a(s)} \dd{s}
= \int_\mathcal{I} \frac{\sqrt{a(s)}}{\sqrt{\hat a_\mathcal{I}}}\sqrt{\hat{a}_{\mathcal{I}}} \dd{s}
< \sqrt{\int_\mathcal{I}\frac{a(s)}{\hat{a}_{\mathcal{I}}}\dd{s}} \sqrt{\int_{\mathcal{I}}\hat{a}_{\mathcal{I}\dd{s}}}
=\sqrt{\int_\mathcal{I}\frac{a(s)}{\hat{a}_{\mathcal{I}}}\dd{s}}\sqrt{\hat{a}_\mathcal{I}\abs{\mathcal{I}}}. \label{eq:C-S}
\end{align}
Note that the inequality above is strict as the equality holds if and only if there exists $c\in\R$ such that $a=c\hat{a}_\mathcal{I}$. If that was the case, then by \eqref{lem:prop of concave hull.2}, $c=1$, and therefore $A=\hat{A}$ on $\mathcal{I}$ which is a contradiction.

Now, by \eqref{lem:prop of concave hull.2} and \eqref{eq:C-S},
\begin{align*}
\int_\mathcal{I} \sqrt{a(s)} \dd{s}
< \underbrace{\sqrt{\frac{\int_\mathcal{I}a(s)\dd{s}}{\hat{a}_{\mathcal{I}}\abs{\mathcal{I}}}}}_{=1} \sqrt{\hat{a}_{\mathcal{I}}}\abs{\mathcal{I}} 
%&& \text{(By \eqref{lem:prop of concave hull.2})}
= \int_\mathcal{I} \sqrt{\hat{a}(s)} \dd{s},
\end{align*}
and the proof is completed.
\end{proof}

The second lemma collects two useful implications from the definition of $\beta_G$. The first one characterizes the $\beta$ such that $\beta\sqrt{a(t)}\leq \sqrt{2\log 2}$ for almost every $t\in\{A\neq \hat{A}\}$, and the second one shows that when $\beta\leq \beta_G$, $\beta\sqrt{\hat{a}(t)}\leq \sqrt{2\log 2}$ for all $t\in\{A\neq \hat{A}\}$.

\begin{lemma}
\label{lem:beta_G}
Suppose that $A$ is non-concave. Then the following statements hold.
\begin{enumerate}[(i)]
\item \label{lem:beta_G.1} $\beta\leq \beta_G$ if and only if the set 
\begin{align*}
\{A\neq\hat{A}\}\cap\{s\in [0,1]:\beta\sqrt{a(s)} > \sqrt{2\log 2}\}
\end{align*}
is of measure zero.
\item \label{lem:beta_G.2} If $\beta\leq \beta_G$, then for every connected component $\mathcal{I}$ of $\{A\neq \hat{A}\}$, we have $\beta\sqrt{\hat{a}_{\mathcal{I}}}\leq \sqrt{2\log 2}$.
\end{enumerate}
\end{lemma}
\begin{proof}
We start with the proof of \eqref{lem:beta_G.1}. By the definition of $\beta_G$, $\beta\leq \beta_G$ is true if and only if almost every $s\in\{A\neq \hat{A}\}$,
\[\beta \sqrt{a(s)} \leq \sqrt{2\log 2}.\]
This immediately implies that the set $\beta\leq \beta_G$ if and only if $\{A\neq\hat{A}\}\cap\{s\in [0,1]:\beta\sqrt{a(s)} > \sqrt{2\log 2}\}$ is of measure zero.

We proceed to the proof of \eqref{lem:beta_G.2}, and our strategy is to prove it by contradiction. Suppose that there exists a connected component $\mathcal{I}$ of $\{A\neq \hat{A}\}$ such that
\begin{align}
\beta \sqrt{\hat{a}_{\mathcal{I}}}> \sqrt{2\log 2}. \label{eq:also bad.3}
\end{align} 
Then,
\begin{align*}
\int_\mathcal{I} \hat{a}_\mathcal{I} \dd{s}
&> \int_\mathcal{I} \frac{2\log 2}{\beta^2} \dd{s} && \text{(By \eqref{eq:also bad.3} and the fact that $\abs{\mathcal{I}}>0$)} \\
&\geq \int_\mathcal{I} a(s) \dd{s} && \text{(By \eqref{lem:beta_G.1} and the assumption that $\beta\leq\beta_G$)} \\
&= \int_\mathcal{I} \hat{a}_\mathcal{I} \dd{s}, && \text{(By \eqref{lem:prop of concave hull.2} of Lemma~\ref{lem:prop of concave hull})}
\end{align*}
which yields a contradiction.
\end{proof}

The third lemma compares the difference between two integrals, one using $a$ and the other using $\hat{a}$.

\begin{lemma}
\label{lem:compare}
Recall that the derivative of $f$ equals
\begin{align}
f'(x) 
=
\begin{cases}
\displaystyle
x, & x < \sqrt{2\log 2} \\
\\
\sqrt{2\log 2}, & x \geq \sqrt{2\log 2}.
\end{cases}
\label{eq:f'}
\end{align} 
\begin{enumerate}[(i)]
\item \label{lem:compare.1} Suppose that $\mathcal{I}$ is a connected component of $\{A\neq \hat{A}\}$. Then for all $\beta\geq 0$,
\begin{align*}
\int_\mathcal{I} \left(f'(\beta \sqrt{\hat{a}(s)}) \sqrt{\hat{a}(s)} - f'(\beta \sqrt{a(s)}) \sqrt{a(s)}\right) \dd{s} \geq 0.
\end{align*}
\item \label{lem:compare.2} Moreover, if $\beta>\beta_G$, there exists a connected component $\mathcal{I}$ of $\{A\neq \hat{A}\}$ such that 
\begin{align*}
\int_\mathcal{I} \left(f'(\beta \sqrt{\hat{a}(s)}) \sqrt{\hat{a}(s)} - f'(\beta \sqrt{a(s)}) \sqrt{a(s)}\right) \dd{s} > 0.
\end{align*}
\end{enumerate}
\end{lemma}
\begin{proof}
We prove this lemma by addressing each point separately. 

\noindent Proof of \eqref{lem:compare.1}. Let $\mathcal{I}$ be a connected component of $\{A\neq \hat{A}\}$. By \eqref{lem:prop of concave hull.2} of Lemma~\ref{lem:prop of concave hull}, $\hat{a}$ is equal to a positive constant $\hat{a}_\mathcal{I}$ on $\text{int}(\mathcal{I})$. We now distinguish the two cases of $\hat{a}_{\mathcal{I}}$.

\noindent \textbf{Case 1: $\beta\sqrt{\hat{a}_\mathcal{I}} < \sqrt{2\log 2}$.} We have
\begin{align*}
& \int_\mathcal{I} \left(f'(\beta \sqrt{\hat{a}(s)}) \sqrt{\hat{a}(s)} - f'(\beta \sqrt{a(s)}) \sqrt{a(s)}\right) \dd{s}  \\
&= \int_\mathcal{I} \left( \beta \hat{a}_{\mathcal{I}} - f'(\beta \sqrt{a(s)}) \sqrt{a(s)}\right) \dd{s} && \text{(By \eqref{eq:f'})} \\
&= \int_\mathcal{I} \left(\beta a(s) - f'(\beta \sqrt{a(s)}) \sqrt{a(s)}\right) \dd{s} && \text{(By \eqref{lem:prop of concave hull.2} of Lemma~\ref{lem:prop of concave hull})} \\
&\geq \int_\mathcal{I} \underbrace{(\beta a(s) - \beta a(s))}_{=0} \dd{s} && \text{(Because $f'(x)\leq x$ for all $x\geq 0$)} \\
&= 0.
\end{align*}

\noindent \textbf{Case 2: $\beta\sqrt{\hat{a}_\mathcal{I}} \geq \sqrt{2\log 2}$.} We have 
\begin{align*}
& \int_\mathcal{I} \left(\sqrt{2\log 2} \sqrt{\hat{a}(s)} - f'(\beta \sqrt{a(s)}) \sqrt{a(s)}\right) \dd{s} \\
&> \int_\mathcal{I} \left(\sqrt{2\log 2} \sqrt{a(s)} - f'(\beta \sqrt{a(s)}) \sqrt{a(s)}\right) \dd{s} && \text{(By \eqref{lem:prop of concave hull.3} of Lemma~\ref{lem:prop of concave hull})} \\
&\geq \int_\mathcal{I} \underbrace{\left(\sqrt{2\log 2} \sqrt{a(s)} - \sqrt{2\log 2} \sqrt{a(s)}\right)}_{=0} \dd{s} && \text{(Because $f'(x)\leq \sqrt{2\log 2}$ for all $x\geq 0$)} \\
&=0.
\end{align*} 
Proof of \eqref{lem:compare.2}. Suppose that $\beta>\beta_G$. We distinguish again the two cases of $\hat{a}_\mathcal{I}$.

\noindent \textbf{Case 1: $\beta\sqrt{\hat{a}_\mathcal{I}} < \sqrt{2\log 2}$.} By Lemma~\ref{lem:beta_G}, there exists a connected component $\mathcal{I}$ of $\{A\neq \hat{A}\}$ such that 
\begin{align}
\abs{\mathcal{I}\cap \{\beta\sqrt{a(s)} > \sqrt{2\log 2}\}}>0. \label{eq:positive measure}
\end{align}
Then we have
\begin{align*}
& \int_\mathcal{I} \left(f'(\beta \sqrt{\hat{a}(s)}) \sqrt{\hat{a}(s)} - f'(\beta \sqrt{a(s)}) \sqrt{a(s)}\right) \dd{s}  \\
&= \int_\mathcal{I} \left( \beta \hat{a}_{\mathcal{I}} - f'(\beta \sqrt{a(s)}) \sqrt{a(s)}\right) \dd{s} && \text{(By \eqref{eq:f'})} \\
&= \int_\mathcal{I} \left(\beta a(s) - f'(\beta \sqrt{a(s)}) \sqrt{a(s)}\right) \dd{s} && \text{(By \eqref{lem:prop of concave hull.2} of Lemma~\ref{lem:prop of concave hull})} \\
&= \int_{\mathcal{I}\cap\{\beta \sqrt{a(s)} \leq \sqrt{2\log 2}\}} \underbrace{(\beta a(s) - \beta a(s))}_{=0} \dd{s} \\
&+ \int_{\mathcal{I}\cap\{\beta \sqrt{a(s)} > \sqrt{2\log 2}\}} \underbrace{(\beta a(s) - \sqrt{2\log 2}\sqrt{a(s)})}_{> 0} \dd{s} >0. && \text{(By \eqref{eq:positive measure})} \\
\end{align*}

\noindent \textbf{Case 2: $\beta\sqrt{\hat{a}_\mathcal{I}} \geq \sqrt{2\log 2}$.}
In this case, as shown in Case 2 in the proof of \eqref{lem:compare.1},
for any connected component $\mathcal{I}$ of $\{A\neq\hat{A}\}$, 
\begin{align*}
& \int_\mathcal{I} \left(f'(\beta \sqrt{\hat{a}(s)}) \sqrt{\hat{a}(s)} - f'(\beta \sqrt{a(s)}) \sqrt{a(s)}\right) \dd{s} > 0.
\end{align*}
This completes the proof.
\end{proof}

We are now ready to prove Proposition~\ref{prop:FE comparison}.

\begin{proof}[Proof of Proposition~\ref{prop:FE comparison}]
In the following, let $\{\mathcal{I}_i\}_{i=1}^\infty$ be the collection of the connected components of $\{A\neq \hat{A}\}$. 

We start with the proof of \eqref{prop:FE comparison.1}. Assume that $\beta\leq \beta_G$. We have
\begin{align*}
G_\beta 
&= \int_0^1 \left( f(\beta \sqrt{\hat{a}(s)}) - f(\beta \sqrt{a(s)}) \right) \dd{s} \\
&= \sum_{i=1}^\infty \int_{\mathcal{I}_i} \left( f(\beta \sqrt{\hat{a}_{\mathcal{I}_i}}) - f(\beta \sqrt{a(s)}) \right) \dd{s} && \text{(By Lemma~\ref{lem:prop of concave hull})}\\
&= \sum_{i=1}^\infty \int_{\mathcal{I}_i} \left( \frac{\beta^2}{2}\hat{a}_{\mathcal{I}_i} - \frac{\beta^2}{2}a(s) \right) \dd{s} && \text{(By \eqref{lem:beta_G.1} and \eqref{lem:beta_G.2} of Lemma~\ref{lem:beta_G})} \\
&= 0, && \text{(By \eqref{lem:prop of concave hull.2} of Lemma~\ref{lem:prop of concave hull})}
\end{align*}
and the proof of \eqref{prop:FE comparison.1} is completed.

We now proceed to the proof of \eqref{prop:FE comparison.2}. Assume that $\beta > \beta_G$. Differentiating $G_\beta$ yields
\begin{align}
G_\beta'
= \int_0^1 \left(f'(\beta \sqrt{\hat{a}(s)}) \sqrt{\hat{a}(s)} - f'(\beta \sqrt{a(s)}) \sqrt{a(s)}\right) \dd{s}. %\label{eq:g' integral}
\end{align}
Again by Lemma~\ref{lem:prop of concave hull}, $G'_\beta$ satisfies the following
\begin{align}
&\int_0^1 \left(f'(\beta \sqrt{\hat{a}(s)}) \sqrt{\hat{a}(s)} - f'(\beta \sqrt{a(s)}) \sqrt{a(s)}\right) \dd{s} \nonumber \\
%&= \int_{\{A=\hat{A}\}} \left(f'(\beta \sqrt{\hat{a}(s)}) \sqrt{\hat{a}(s)} - f'(\beta \sqrt{a(s)}) \sqrt{a(s)}\right) \dd{s} \nonumber \\
&= \sum_{i=1}^\infty \int_{\mathcal{I}_i} \left(f'(\beta \sqrt{\hat{a}_{\mathcal{I}_i}}) \sqrt{\hat{a}_{\mathcal{I}_i}} - f'(\beta \sqrt{a(s)}) \sqrt{a(s)}\right) \dd{s} > 0, \label{eq:g' integral two terms}
\end{align}
where \eqref{eq:g' integral two terms} is true because of Lemma~\ref{lem:compare} and the assumption that $\beta > \beta_G$. This proves \eqref{prop:FE comparison.2}.
\end{proof}

\section{Proof of Lemma~\ref{lem:sandwich}} \label{ann:proof of lem:sandwich}

This section is devoted to the proof of Lemma~\ref{lem:sandwich}, and the strategy is to compare the free energy of the CREM with the free energy of the branching random walk, which is defined as follows. 
\paragraph{The branching random walk.} Let $(\tilde{X}_u)_{u\in\T_M}$ be a centered Gaussian process indexed by $\T_M$ with the covariance function
\begin{align*}
\EX{\tilde{X}_u\tilde{X}_w} = \abs{u\wedge w}
\end{align*}
for all $u,w\in\T_M$. This Gaussian process is called the branching random walk with standard Gaussian increments, which will be abbreviated as the branching random walk. Define
\begin{align*}
f_M(\beta) = \frac{1}{M} \EX{\log \tilde{Z}_{\beta,M}}, \quad \text{where} \quad \tilde{Z}_{\beta,M} = \sum_{\abs{u}=M} e^{\beta \tilde{X}_u}.
\end{align*}
It is known that (see, \cite{chauvinBoltzmannGibbsWeightsBranching1997}) the function $\eqref{eq:BRW Fe}$ is the pointwise limit of $f_M(\beta)$, i.e., for all $\beta\in [0,\infty)$,
\begin{align*}
f = \lim_{M\rightarrow\infty} f_M(\beta).
\end{align*}
The proof of Lemma~\ref{lem:sandwich} relies on a quantitative estimate of the convergence above, which is stated in detail in Lemma~\ref{lem:uniform}. Before we proceed to Lemma~\ref{lem:uniform}, we state the following lemma that is handy to prove Lemma~\ref{lem:uniform}.

\begin{lemma}
\label{lem:prop of g_M}
Define $g_M:[0,\infty)\rightarrow\R$ as $g_M(0) \coloneqq 0$ and $g_M(\beta) \coloneqq f_M(\beta)/\beta - 2\log 2/\beta$. Define $g(\beta) \coloneqq f(\beta)/\beta - 2\log 2/\beta$ which equals
\begin{align*}
g(\beta)\coloneqq
\begin{cases}
\displaystyle
\frac{\beta}{2}, & \beta \in [0,\sqrt{2\log 2}] \\
\displaystyle
\sqrt{2\log 2} - \frac{\log 2}{\beta}, & \beta > \sqrt{2\log 2}
\end{cases}
\end{align*}
Then, the following statements are true.
\begin{enumerate}[(i)]
\item For all $\beta\in [0,\infty)$, $\lim_{M\rightarrow\infty} g_M(\beta)=g(\beta)$. \label{g_M -> g}
%\item The function $g$ is continuous and non-decreasing. Moreover, $g(0)=0$ and $\lim_{\beta\rightarrow\infty} g(\beta)=\sqrt{2\log 2}$.
\item For all $\beta\in [0,\infty)$ and $M\in\N$, $g_M(\beta)\leq g(\beta)$. \label{g_M small}
\item For all $M\in\N$, the function $g_M$ is non-decreasing. Moreover, $g_M(\infty)\coloneqq \lim_{\beta\rightarrow\infty} g_M(\beta)$ exists and $\lim_{M\rightarrow\infty} g_M(\infty) = \sqrt{2\log 2} = g(\infty)$, where $g(\infty)\coloneqq \lim_{\beta\rightarrow\infty} g(\beta)$.  \label{g_M limit exists}
\item The sequence of functions $g_M$ converges uniformly to $g$. \label{g_M -> g uniformly}
\end{enumerate}
\end{lemma}

\begin{remark}
As we will see below, the proof of \eqref{g_M small} in Lemma~\ref{lem:prop of g_M} is a standard argument in the context of statistical physics. The argument to prove \eqref{g_M -> g uniformly} is a slight modification of the proof of Dini's second theorem\footnote{This simple but handy result appears in some French textbooks under the name ``deuxi\`{e}me th\'{e}or\`{e}me de Dini''. One can find the proof in Solution 127 in Part II, Chapter 3 of \cite{polyaProblemsTheoremsAnalysis1972}.} which states that if a sequence of monotone (continuous or discontinuous) functions converges on a closed interval to a continuous function, the sequence converges uniformly. The second statement of \eqref{g_M limit exists} allows us to generalize Dini's second theorem to our setting.
\end{remark}

\begin{proof}[Proof of Lemma~\ref{lem:prop of g_M}]
We prove this lemma by addressing each point separately.

\noindent Proof of \eqref{g_M -> g}. This follows directly from the definition of $g_M$ and $g$, and the pointwise convergence of $f_M$ to $f$.

\noindent Proof of \eqref{g_M small}. It suffices to show that for all $\beta\in [0,\infty)$ and $M\in\N$, $f_M(\beta)\leq f(\beta)$. To this purpose, we claim that for all $\beta\in [0,\infty)$ the sequence $Mf_M$ is super-additive. If this is true, then Fekete's lemma implies that $f_M(\beta) \leq f(\beta)$.

Now, fixing $M_1,M_2\in\N$, we have 
\begin{align*}
\EX{\log \tilde{Z}_{\beta,M_1+M_2}}
&= \EX{\log \sum_{\abs{u}=M_1+M_2} e^{\beta \tilde{X}_{u}}} \\
&= \EX{\log \sum_{\abs{u_1}=M_1}\sum_{\abs{u_2}=M_2} e^{\beta (\tilde{X}_{u_1}+\tilde{X}_{u_2}^{u_1})}} \\
&= \EX{\log \sum_{\abs{u_1}=M_1}e^{\beta \tilde{X}_{u_1}} \tilde{Z}_{\beta,M_2}^{u_1}} \\
&= \EX{\log \tilde{Z}_{\beta,M_1}} + \EX{\log \sum_{\abs{u_1}=M_1}\frac{e^{\beta X_{u_1}}}{\tilde{Z}_{\beta,M_1}} \tilde{Z}_{\beta,M_2}^{u_1}}.
\end{align*}
By Jensen's inequality and the branching property, we have 
\begin{align*}
\EX{\log \sum_{\abs{u_1}=M_1}\frac{e^{\beta X_{u_1}}}{\tilde{Z}_{\beta,M_1}} \tilde{Z}_{\beta,M_2}^{u_1}} \geq \EX{ \sum_{\abs{u_1}=M_1}\frac{e^{\beta X_{u_1}}}{\tilde{Z}_{\beta,M_1}} \log \tilde{Z}_{\beta,M_2}^{u_1}} = \EX{\log \tilde{Z}_{\beta,M_2}}.
\end{align*}
Therefore, we conclude that 
\begin{align*}
(M_1+M_2)g_{M_1+M_2} 
&= \EX{\log \tilde{Z}_{\beta,M_1+M_2}} \\
&\geq \EX{\log \tilde{Z}_{\beta,M_1}} + \EX{\log \tilde{Z}_{\beta,M_2}}
= M_1 g_{M_1} + M_2 g_{M_2}.
\end{align*}

\noindent Proof of \eqref{g_M limit exists}. Fix $M\in\N$. For all $\beta>0$, by Jensen's inequality and the fact that $x\mapsto \log x$ is concave,
\begin{align*}
\log \sum_{\abs{u}=M} \frac{1}{2^M} e^{\beta \tilde{X}_u} 
\geq \sum_{\abs{u}=M} \frac{1}{2^M} \log e^{\beta \tilde{X}_u}
= \sum_{\abs{u}=M} \frac{1}{2^M} \beta \tilde{X}_u.
\end{align*}
Thus, by the fact that $(\tilde{X}_u)_{\abs{u}=M}$ is centered,
\begin{align*}
g_M(\beta)
\geq \frac{1}{\beta M}\EX{\log \sum_{\abs{u}=M} \frac{1}{2^M} e^{\beta \tilde{X}_u}}
\geq \EX{\sum_{\abs{u}=M} \frac{1}{2^M} \beta \tilde{X}_u}
= 0 = g(0).
\end{align*}
For all $0<\beta < \beta'$, the function $x\mapsto x^{\beta'/\beta}$ is convex. Therefore, by Jensen's inequality,
\begin{align*}
\left(\sum_{\abs{u}=M}\frac{1}{2^M} e^{\beta \tilde{X}_u}\right)^{\beta'/\beta} \leq \sum_{\abs{u}=M}\frac{1}{2^M} e^{\beta' \tilde{X}_u}.
\end{align*}
It then yields immediately that
\begin{align*}
g_M(\beta) 
= \frac{1}{\beta M} \EX{\log \left(\sum_{\abs{u}=M}\frac{1}{2^M} e^{\beta \tilde{X}_u}\right)} 
\leq \frac{1}{\beta' M} \EX{\log \left(\sum_{\abs{u}=M}\frac{1}{2^M} e^{\beta' \tilde{X}_u}\right) }
= g_M(\beta'),
\end{align*}
which proves that $g_M$ is non-decreasing. Now, by \eqref{g_M small} and the fact that $g\leq \sqrt{2\log 2}$, monotone convergence theorem implies that $\lim_{\beta\rightarrow\infty} g_M(\beta)$ exists and is bounded from above by $\sqrt{2\log 2}$. Finally, note that 
\begin{align*}
\frac{1}{M}\EX{\max_{\abs{u}=M}X_u} - \frac{\log 2}{\beta} \leq g_M(\beta) \leq \frac{1}{M}\EX{\max_{\abs{u}=M}X_u}.
\end{align*}
Taking $\beta\rightarrow\infty$, we obtain the equality
\begin{align*}
g_M(\infty) = \frac{1}{M}\EX{\max_{\abs{u}=M}X_u}.
\end{align*}
It is well-known that
\begin{align*}
\lim_{M\rightarrow\infty}\frac{1}{M}\EX{\max_{\abs{u}=M}X_u} = \sqrt{2\log 2},
\end{align*}
which is an implication of Theorem 3.1 in \cite{CREM04} by letting the covariance function to be the identity function. Therefore, taking $M\rightarrow\infty$, we conclude that
\begin{align*}
\lim_{M\rightarrow\infty} g_M(\infty) = \sqrt{2\log 2} = \lim_{\beta\rightarrow\infty} g(\beta) = g(\infty).
\end{align*}

\noindent Proof of \eqref{g_M -> g uniformly}. Fix $\varepsilon>0$. By its definition, the function $g$ is continuous and non-decreasing on $[0,\infty)$. Moreover, by \eqref{g_M limit exists}, $g(\infty)\coloneqq \lim_{\beta\rightarrow\infty} g(\beta)$ and $g_M(\infty)\coloneqq \lim_{\beta\rightarrow\infty} g_M(\beta)$ exist and $\lim_{M\rightarrow\infty} g_M(\infty)=g(\infty)$. By the intermediate value theorem, there exists a subdivision $0=\beta_0<\beta_1<\cdots<\beta_{k-1}<\beta_k=\infty$ such that $g(\beta_{i+1})-g(\beta_i)<\varepsilon$, for all $i=0,\ldots,k-1$. Thus, for all $\beta\in [0,\infty)$ and $i=0,\ldots,k-1$, we have
\begin{align}
g_M(\beta) - g(\beta) 
&\leq g_M(\beta_{i+1}) - g(\beta_i) && \text{(Because $g_M$ and $g$ are non-decreasing)} \nonumber \\
&\leq g_M(\beta_{i+1}) - g(\beta_{i+1}) + \varepsilon && \text{(By the choice of subdivision)} \label{eq:g_M(beta) - g(beta).1} 
\end{align}
and 
\begin{align}
g_M(\beta) - g(\beta) 
&\geq g_M(\beta_{i}) - g(\beta_{i+1}) && \text{(Because $g_M$ and $g$ are non-decreasing)}\nonumber \\
&\geq g_M(\beta_{i}) - g(\beta_{i}) - \varepsilon. && \text{(By the choice of subdivision)}
\label{eq:g_M(beta) - g(beta).2}   
\end{align}
By \eqref{g_M -> g} and \eqref{g_M limit exists}, there exists $M_\varepsilon\in\N$ such that for all $M\geq M_\varepsilon$ and $i=0,\ldots,k-1$,
\begin{align}
\abs{g_M(\beta_i) - g(\beta_i)} < \varepsilon. \label{eq:|g_M(beta_i) - g_M(beta_i)| < varepsilon}
\end{align}
Combining \eqref{eq:g_M(beta) - g(beta).1}, \eqref{eq:g_M(beta) - g(beta).2} and \eqref{eq:|g_M(beta_i) - g_M(beta_i)| < varepsilon}, we conclude that for all $M\geq M_\varepsilon$,
\begin{align*}
\abs{g_M(\beta) - g(\beta)} < 2\varepsilon,
\end{align*}
which proves that $g_M$ converges to $g$ uniformly as, desired.
\end{proof}

Lemma~\ref{lem:prop of g_M} implies the following quantitative convergence of $f_M$.

\begin{lemma}
\label{lem:uniform}
For all $\varepsilon>0$, there exists $M\in\N$ independent of $\beta\in [0,\infty)$ such that
\begin{align}
\abs{f_M(\beta)-f(\beta)} \leq \beta \varepsilon. \label{eq:uniform}
\end{align}
\end{lemma}
\begin{proof}
By \eqref{g_M -> g uniformly} of Lemma~\ref{lem:prop of g_M}, $g_M$ converges uniformly to $g$ on $[0,\infty)$. Combining this with the fact
\begin{align*}
\abs{f_M(\beta)-f(\beta)} = \beta \abs{g_M(\beta)-g(\beta)}.
%\leq \beta_0 \abs{g_M(\beta)-g(\beta)}, \quad \forall \beta\in [0,\beta_0],
\end{align*}
the proof is completed.
\end{proof}

We now proceed to the proof of Lemma~\ref{lem:sandwich}.

\begin{proof}[Proof of Lemma~\ref{lem:sandwich}]
Fix $M_N$ a sequence such that $M_N\in\llbracket 1,N\rrbracket$ and $M_N\rightarrow\infty$.

For all $N\in\N$, by Kahane's inequality (see, Theorem 3.11 in \cite{ledouxProbabilityBanachSpaces1991}), for all $k\in\llbracket 0,\floor{N/M_N}\rrbracket$, 
\begin{align}
\EX{\log \tilde{Z}_{\beta\sqrt{a}^-_k,M_N}}
\leq \EX{\log Z^{(kM_N)}_{\beta,M_N}}
\leq \EX{\log \tilde{Z}_{\beta\sqrt{a}^+_k,M_N}},
\label{eq:first sandwich}
\end{align}
where
\begin{align*}
a_k^- \coloneqq \essinf_{t\in [\frac{kM_N}{N},\frac{(k+1)M_N}{N}]} a(t)
\quad \text{and}
\quad a_k^+ \coloneqq \esssup_{t\in [\frac{kM_N}{N},\frac{(k+1)M_N}{N}]} a(t).
\end{align*}

Now, fix $\varepsilon>0$. By Lemma~\ref{lem:uniform}, there exists $N_0\in\N$ such that for all $N\geq N_0$,
\begin{align}
\frac{1}{M_N} \EX{\log \tilde{Z}_{\beta\sqrt{a}^-_k,M}}
\geq f(\beta\sqrt{a}^-_k) - \varepsilon\beta\sqrt{a}^-_k
\quad
\text{and}
\quad
\frac{1}{M_N} \EX{\log \tilde{Z}_{\beta\sqrt{a}^+_k,M}}
\leq f(\beta\sqrt{a}^+_k) + \varepsilon\beta\sqrt{a}^+_k. \label{eq:approx ineq}
\end{align}
Combining \eqref{eq:first sandwich} and \eqref{eq:approx ineq}, for all $N\geq N_0$, we conclude that 
\begin{align*}
\EX{\log Z^{(kM_N)}_{\beta,M_N}}
\geq
\EX{\log \tilde{Z}_{\beta\sqrt{a}^-_k,M_N}}
\geq
f(\beta\sqrt{a}^-_k) - \varepsilon\beta\sqrt{a}^-_k
\end{align*}
and 
\begin{align*}
\EX{\log Z^{(kM_N)}_{\beta,M_N}}
\leq \EX{\log \tilde{Z}_{\beta\sqrt{a}^+_k,M_N}}
\leq f(\beta\sqrt{a}^+_k) + \varepsilon\beta\sqrt{a}^+_k.
\end{align*}
These complete the proof.
\end{proof}

\paragraph{Acknowledgments.} I want to thank Pascal Maillard for his guidance throughout the whole project. I am also grateful to Alexandre Legrand and Michel Pain for stimulating discussions.

\printbibliography
\end{document}